\documentclass[11pt]{article}
\usepackage{amsfonts}
\usepackage{amsthm}
\input{epsf.sty}
\usepackage{latexsym,amsmath,amsfonts,amscd}
\usepackage{epsfig}
\usepackage{changebar}
\usepackage{pstricks}
\usepackage{pst-plot}
\usepackage{multirow}
\usepackage{subfigure}
\usepackage{placeins}
\usepackage{amssymb}
\usepackage{mathrsfs}
\usepackage[title]{appendix}
\usepackage{bm}
\usepackage{hyperref}
\usepackage{geometry}
\usepackage{multirow,bigstrut}
\usepackage{enumitem}
\usepackage{booktabs}

\numberwithin{equation}{section}

\theoremstyle{plain}
\newtheorem{thm}{Theorem}[section]

\newtheorem{prop}[thm]{Proposition}

\newtheorem{cor}[thm]{Corollary}
\newtheorem{de}[thm]{Definition}
\newtheorem{rem}[thm]{Remark}

\theoremstyle{definition}
\newtheorem{exam}[thm]{Example}

\newcommand{\brac}[1]{\left(#1\right)}

\DeclareMathOperator{\diag}{diag}

\allowdisplaybreaks[4]
\begin{document}
\baselineskip=1.6pc

\vspace{.5in}

\begin{center}

{\large\bf Machine learning moment closure models for the radiative transfer equation II: enforcing  global hyperbolicity in gradient based closures}

\end{center}

\vspace{.1in}

\centerline{
Juntao Huang \footnote{Department of Mathematics,
Michigan State University, East Lansing, MI 48824, USA.
E-mail: huangj75@msu.edu} \quad
Yingda Cheng  \footnote{Department of Mathematics, Department of Computational Mathematics, Science and Engineering, Michigan State University, East Lansing, MI 48824, USA. E-mail: ycheng@msu.edu. Research is supported by NSF grants    DMS-2011838 and AST-2008004.} \quad
Andrew J. Christlieb \footnote{Department of Computational Mathematics, Science and Engineering, Michigan State University, East Lansing, Michigan 48824, USA. E-mail: christli@msu.edu. Research is supported by: AFOSR grants FA9550-19-1-0281 and FA9550-17-1-0394; NSF grants DMS-1912183 and AST-2008004; and DoE grant DE-SC0017955.} \quad
Luke F. Roberts \footnote{National Superconducting Cyclotron Laboratory and Department of Physics and Astronomy, Michigan State University, East Lansing, MI 48824, USA. E-mail: robertsl@nscl.msu.edu. Research is supported by: NSF grant AST-2008004; and DoE grant DE-SC0017955.} \quad
Wen-An Yong \footnote{Department of Mathematical Sciences, Tsinghua University, Beijing 100084, China.
E-mail: wayong@tsinghua.edu.cn}
}

\vspace{.4in}

\centerline{\bf Abstract}
\vspace{.1in}

This is the second paper in a series in which we develop machine learning (ML) moment closure models for the radiative transfer equation (RTE). In our previous work \cite{huang2021gradient}, we proposed an approach to directly learn the gradient of the unclosed high order moment, which performs much better than learning the moment itself and the conventional $P_N$ closure. However, the ML moment closure model in \cite{huang2021gradient} is not able to guarantee hyperbolicity and long time stability. We propose in this paper a method to enforce the global hyperbolicity of the ML closure model. The main idea is to seek a symmetrizer (a symmetric positive definite matrix) for the closure system, and derive constraints such that the system is globally symmetrizable hyperbolic. It is shown that the new ML closure system inherits the dissipativeness of the RTE and preserves the correct diffusion limit as the Knunsden number goes to zero. Several benchmark tests including the Gaussian source problem and the two-material problem show the good accuracy, long time stability and generalizability of our globally hyperbolic ML closure model.

\bigskip

\bigskip



\pagenumbering{arabic}
\newpage
\section{Introduction}\label{sec:intro}
\setcounter{equation}{0}
\setcounter{figure}{0}
\setcounter{table}{0}
In this paper we introduce an  extension to our previous work on ML closures for radiative transfer  modeling \cite{huang2021gradient}. The new approach offers long time stability by ensuring mathematical consistency between the closure and the macroscopic model.  Further, the numerical results in section 4 demonstrate the plausibility of capturing kinetic effects in a moment system with a handful of moments and an appropriate closure model.  

The study of radiative transfer is of vital importance in many fields of science and engineering including astrophysics \cite{pomraning2005equations}, heat transfer \cite{koch2004evaluation}, and optical imaging \cite{klose2002optical}. The fundamental equation describing radiative transfer is an integro-differential equation, termed as the radiative transfer equation (RTE). Nevertheless, in most cases, any mesh based numerical discretization of the RTE equation leads to unacceptable computational costs due to the curse of dimensionality.

Moment methods study the evolution of a finite number of moments of the specific intensity in the RTE. Typically, the evolution of a lower-order moment depends on a higher-order moment, leading to what is known as the moment closure problem. Hence, one has to introduce suitable closure relations that relates the highest order moment included with the lower order moments in order to get a closed system of equations.  Many moment closure models have been developed, including the $P_N$ model \cite{chandrasekhar1944radiative}; the variable Eddington factor models \cite{levermore1984relating,murchikova2017analytic}; the entropy-based $M_N$ models \cite{hauck2011high,alldredge2012high,alldredge2014adaptive};  the positive $P_N$ models \cite{hauck2010positive}; the filtered $P_N$ models \cite{mcclarren2010robust,laboure2016implicit}; the $B_2$ models \cite{alldredge2016approximating}; and the $MP_N$ model \cite{fan2020nonlinear,fan2020nonlinear2,li2021direct}. 

In moment closure problems,   hyperbolicity is a critical issue, which is essential for a system of first-order partial differential equations (PDEs) to be well-posed \cite{serre1999systems}. The pioneering work on the moment closure for the gas kinetic theory by Grad in \cite{grad1949kinetic} is the most basic one among the moment models. However, it was discovered in \cite{cai2014hyperbolicity} that the equilibrium is on the boundary of the hyperbolicity region for the Grad's 13-moment model in the three-dimensional case. Due to such a deficiency, the application of the moment method is severely restricted. This issue also attracts a lot of attention, with many papers in the literature focusing on the development of globally hyperbolic moment systems \cite{cai2014globally,fan2020nonlinear,fan2020nonlinear2,li2021direct}. 

The traditional trade off in introducing a closure relation and solving a moment model instead of a kinetic equation is generic accuracy verses practical computability.
However, thanks to the rapid development of machine learning (ML)  and data-driven modeling \cite{brunton2016discovering,raissi2019physics,han2018solving}, a new approach to solve the moment closure problem has emerged based on ML \cite{han2019uniformly,scoggins2021machine,huang2020learning,bois2020neural,ma2020machine,wang2020deep,maulik2020neural,huang2021gradient}. This approach offers a path for multi-scale problems that is relatively unique, promising to capture kinetic effects in a  moment model with only a handful of moments.  For the detailed literature review, we refer readers to \cite{huang2021gradient}. We only remark that most of the works mentioned above are not able to guarantee hyperbolicity or long time stability, except the work in \cite{huang2020learning}. In \cite{huang2020learning}, based on the conservation-dissipation formalism \cite{zhu2015conservation} of irreversible thermodynamics, the authors proposed a stable ML closure model with hyperbolicity and Galilean invariance for the Boltzmann BGK equation. Nevertheless, their model is limited to one extra non-equilibrium variable and it is still not clear how to generalize to arbitrary number of moments.

The work in this paper is a continuation of the previous work in \cite{huang2021gradient}, where we proposed to directly learn a closure that relates the gradient  of the   highest order moment to gradients of lower order moments. This new approach is consistent with the exact closure we derived for the free streaming limit and also provides a natural output normalization \cite{huang2021gradient}. A variety of numerical tests show that the ML closure model in \cite{huang2021gradient} has better accuracy than an ML closure based on learning a  relation between the moments, as opposed to a relation between the gradients, and the conventional $P_N$ closure \cite{huang2021gradient}. However, it is not able to guarantee hyperbolicity and long time simulations are not always satisfactory. Consequently, the focus of this work is to develop a method to enforce global hyperbolicity of the ML closure model.
The main idea  is motivated by the  observation that the coefficient matrix for the $P_N$ closure is a tridiagonal matrix with positive off-diagonal entries. This indicates the existence of a diagonal symmetrizer matrix such that the $P_N$ closure is symmetrizable hyperbolic. 
Motivated by this observation, we propose to only keep the last several components in the ansatz, see equation \eqref{eq:anstaz-learn-gradient} in Section \ref{sec:sub-hyperbolic-closure}, and seek a symmetrizer, which is a symmetric positive definite (SPD) matrix, of a block diagonal form. For degrees of freedom no larger than four in the ansatz, relating the gradient of the $(N+1)^{th}$ moment to the gradient of the $N^{th}$, $(N-1)^{th}$, $(N-2)^{th}$, and $(N-3)^{th}$ moments, we derive constraints to guarantee the hyperbolicity of the ML closure model, see Theorem \ref{thm:hyperbolic} in Section \ref{sec:sub-hyperbolic-closure}. Moreover, due to the block diagonal structure of the symmetrizer, we show that the ML closure system also satisfies the structural stability condition in \cite{yong1999singular} when taking into account the relaxation effects of the source terms, see Theorem \ref{thm:structure-stability} in Section \ref{sec:sub-hyperbolic-closure}. This condition characterises the dissipation feature of the moment closure system and is analogous to the H-theorem for the Boltzmann equation. We also consider our moment closure system under a diffusive scaling and formally show that our model preserves the correct diffusion limit as the Knunsden number goes to zero. The justification of the formal asymptotic analysis could be rigorously verified \cite{peng2016uniform,lattanzio2001hyperbolic}, with the aid of the structural stability condition. We also remark that, unlike \cite{huang2020learning}, the work in this paper constructs ML closure models that are hyperbolic with an arbitrary number of moments. We numerically tested that the hyperbolic ML closure model has good accuracy in a variety of numerical examples. Moreover, the non-hyperbolic ML closure model in \cite{huang2021gradient} blows up for long time simulations, while the hyperbolic one remains stable and has good accuracy.

The remainder of this paper is organized as follows. In Section \ref{sec:moment-method}, we introduce the hyperbolic ML moment closure model and discuss the diffusion limit of the model. In Section \ref{sec:training}, we present the details in the training of the neural networks. The effectiveness of our ML closure model is demonstrated through extensive numerical results in Section \ref{sec:numerical-test}. Some concluding remarks are given in Section \ref{sec:conclusion}.

\section{ML moment closure for radiative transfer equation}\label{sec:moment-method}

In this section, we first review the ML   moment closure method for the RTE in slab geometry proposed in \cite{huang2021gradient} and propose our approach to enforce the hyperbolicity of the ML moment closure method. We then formally show that the resulting closure model can capture the correct diffusion limit.

\subsection{Hyperbolic ML closure model}\label{sec:sub-hyperbolic-closure}

We consider the time-dependent RTE for a gray medium in slab geometry:
\begin{equation}\label{eq:rte}
	\partial_t f + v \partial_x f = {\sigma_s}\brac{\frac{1}{2} \int_{-1}^1 f dv - f} - \sigma_a f, \quad -1\le v\le 1
\end{equation}
Here, $f=f(x,v,t)$ is the specific intensity of radiation. The variable $v\in[-1, 1]$ is the cosine of the angle between the photon velocity and the $x$-axis. $\sigma_s = \sigma_s(x)\ge 0$ and $\sigma_a = \sigma_a(x)\ge 0$ are the scattering and absorption coefficients.

Denote the $k$-th order Legendre polynomial by $P_k = P_k(x)$. Define the $k$-th order moment by
\begin{equation}
	m_k(x,t) = \frac{1}{2} \int_{-1}^1 f(x,v,t) P_k(v) dv.
\end{equation}
Multiplying by $P_k(v)$ on both sides of \eqref{eq:rte} and integrating over $v\in[-1, 1]$, we derive the moment equations:
\begin{equation}
\begin{aligned}
	 \partial_t m_0 + \partial_x m_1 &= -  \sigma_a m_0 \\
	 \partial_t m_1 + \frac{1}{3} \partial_x m_0 + \frac{2}{3} \partial_x m_2  &= - ( \sigma_s +  \sigma_a) m_1 \\
	& \cdots \\
	 \partial_t m_N + \frac{N}{2N+1} \partial_x m_{N-1} + \frac{N+1}{2N+1} \partial_x m_{N+1}  &= - ( \sigma_s +  \sigma_a) m_N \\
\end{aligned}
\end{equation}

The learned gradient approach proposed in \cite{huang2021gradient} is to find  a relation between $\partial_x m_{N+1}$ and the gradients on lower order moments:
\begin{equation}\label{eq:anstaz-learn-gradient}
	\partial_x m_{N+1} = \sum_{i=0}^N \mathcal{N}_i(m_0,m_1,\cdots,m_N) \partial_x m_i
\end{equation}
with $\mathcal{N}=(\mathcal{N}_0, \mathcal{N}_1, \cdots, \mathcal{N}_N):\mathbb{R}^{N+1}\rightarrow\mathbb{R}^{N+1}$ approximated by a neural network and learned from data. In this way, we obtain a quasi-linear first-order systems. This approach is shown to have uniform accuracy in the optically thick regime, intermediate regime and the optically thin regime. Moreover, the accuracy of this gradient-based model is much better than the approach based on creating a ML closure directly trained to match the moments, as well as the conventional $P_N$ closure. However, the learned gradient model exhibits numerical instability due to the loss of hyperbolicity \cite{huang2021gradient}. This severely restricts the application of this model, especially for long time simulations.

The main idea of this work is to enforce the hyperbolicity  by demanding that the coefficient matrix of the closure system $A$ is real diagonizable. For this purpose, we seek a SPD matrix $A_0$ (also called a symmetrizer) such that $A_0 A$ is a symmetric matrix. Namely, the system is symmetrizable hyperbolic, see the details in Appendix \ref{sec:appendix-hyperbolic}. This places certain constraints on the functions $\mathcal{N}_i$.

Plugging \eqref{eq:anstaz-learn-gradient} into the moment closure system, we have:
\begin{equation*}
	\partial_t m_N + \frac{N}{2N+1} \partial_x m_{N-1} + \frac{N+1}{2N+1} \brac{\sum_{k=0}^N \mathcal{N}_k(m_0,m_1,\cdots,m_N) \partial_x m_k}  = - ( \sigma_s +  \sigma_a) m_N.
\end{equation*}
Then, we can write down the coefficient matrix of this moment closure system:
\begin{equation}\label{eq:jacobi-matrix}
	A = 
	\begin{pmatrix}
    0 				& 	1 				&	0  			& 0  	&	\dots 	& 	0 	\\
    \frac{1}{3} 	& 	0 				& \frac{2}{3} 	& 0  	&	\dots 	& 	0	\\
    0 & \frac{2}{5} & 	0 				& \frac{3}{5} 	& \dots & 0 \\
    \vdots & \vdots &   \vdots 			& \ddots & \vdots & \vdots \\
    0 & 0 & 	\dots 				& \frac{N-1}{2N-1} 	& 0 & \frac{N}{2N-1} \\
    a_0 & a_1 & \dots & a_{N-2} & a_{N-1} & a_N
	\end{pmatrix}
\end{equation}
with
\begin{equation}
\label{acoef}
	a_j = 
	\left\{
	\begin{aligned}
	& \frac{N+1}{2N+1}\mathcal{N}_j, \quad & j\ne N-1 \\
	& \frac{N}{2N+1} + \frac{N+1}{2N+1}\mathcal{N}_j, \quad & j=N-1
	\end{aligned}
	\right.
\end{equation}

We first observe that, for the $P_N$ closure, i.e. when $\mathcal{N}_j=0$ for $0\le j\le N$, $A$ is a tridiagonal matrix with positive off-diagonal entries. In this case, $A_0$ can be taken as a diagonal matrix:
\begin{equation}
	A_0 = \diag(1, 3, 5, \dots, 2N+1),
\end{equation}
and $A_0 A$ is a tridiagonal symmetric matrix:
\begin{equation}
	A_0 A =
	\begin{pmatrix}
    0 	& 	1 	&	0  	& 	0  	&	\dots 	& 	0 	\\
    1	& 	0 	& 	2 	& 	0  	&	\dots 	& 	0	\\
    0 	&	2 	& 	0 	&	3	&	\dots & 0 \\
    \vdots & \vdots &   \vdots 			& \ddots & \vdots & \vdots \\
    0 	& 	0 	& 	\dots 				& N-1 	& 0 & N \\
    0 	& 	0 	& \dots & 0 & N & 0
	\end{pmatrix}	
\end{equation}

Inspired by this observation, we propose to first investigate the case when the ansatz is given by only   the last $(k+1)$ components 
\begin{equation}\label{eq:anstaz-partial-learn-gradient}
	\partial_x m_{N+1} = \sum_{i=N-k}^N \mathcal{N}_i(m_0,m_1,\cdots,m_N) \partial_x m_i,
\end{equation}
where $k$ is a parameter and $k \le N$. We will find  a SPD matrix $A_0$ of the form
\begin{equation}
	A_0 =
	\begin{pmatrix}
    D 	& 	0 	\\
	0 	& 	B 	
	\end{pmatrix}	
\end{equation}
with $D=\diag(1,3,\cdots,2(N+1-k)-1)\in\mathbb{R}^{(N+1-k)\times(N+1-k)}$ and $B\in\mathbb{R}^{k\times k}$ being a SPD matrix. 
With some algebraic calculation (details given in the appendix), when $k=3$, we can derive the constraints explicitly as shown in the following theorem.
\begin{thm}\label{thm:hyperbolic}
	Consider  matrix $A\in\mathbb{R}^{(N+1)\times(N+1)}$ with $N\ge3$ of the form
	\begin{equation}\label{eq:thm-jacobi-matrix}
		A = 
		\begin{pmatrix}
	    0 				& 	1 & 0		& 0				&	0  			& 0  	&	\dots 	& 	0 	\\
	    \frac{1}{3} 	& 	0 & \frac{2}{3} & 0		& 0	& 0  	&	\dots 	& 	0	\\
	    0 & \frac{2}{5} & 	0 & \frac{3}{5} & 0		& 0	& \dots & 0 \\
	    \vdots & \vdots &   \vdots 	&   \vdots	&   \vdots	& \ddots & \vdots & \vdots \\
		0 & 0 & 	\dots 	& \frac{N-3}{2N-5} 	& 0 & \frac{N-2}{2N-5} & 0 & 0\\
	    0 & 0 & 	\dots 	& 0		& \frac{N-2}{2N-3} 	& 0 & \frac{N-1}{2N-3} & 0 \\
	    0 & 0 & 	\dots 	& 0		& 0	& \frac{N-1}{2N-1} 	& 0 & \frac{N}{2N-1} \\
	    0 & 0 & \dots & 0 & a_{N-3} & a_{N-2} & a_{N-1} & a_N
		\end{pmatrix}
	\end{equation}
	where $a_i$ has been specified in \eqref{acoef}.
	If the coefficients $a_i$ for $i=N-3,N-2,N-1,N$ satisfy the following constraints:
	\begin{equation}\label{eq:constraint-dof4}
	\begin{aligned}
		& a_{N-3} > -\frac{(N-1)(N-2)}{N(2N-3)}, \\
		& a_{N-1} > \frac{g(a_{N-3},a_{N-2},a_{N};N)}{(N-2) (a_{N-3}(2N-3)N + (N-1)(N-2))^2}
	\end{aligned}
	\end{equation}
	where $g=g(a_{N-3},a_{N-2},a_{N};N)$ is a function given by
	\begin{equation*}
	\begin{aligned}
		g={}& a_{N-3}^3 (N-1) N^2 (3-2 N)^2 + a_{N-2} (2 N-1) (N-2)^3 (a_{N-2} N-a_N (N-1))\\
		&+a_{N-3} (N-2)^2 (a_N (4N^2-8N+3) (a_{N-2} N-a_N (N-1))+(N-1)^3)\\
		&+2 a_{N-3}^2 (N-1)^2 N (2 N-3) (N-2),
	\end{aligned}
	\end{equation*}
	then there exist a SPD matrix $A_0=\diag(D, B)\in\mathbb{R}^{(N+1)\times(N+1)}$ such that $A_0 A$ is symmetric. Here, $D=\diag(1,3,5,\cdots,2N-5)\in\mathbb{R}^{(N-2)\times(N-2)}$ and $B\in\mathbb{R}^{3\times3}$ is a SPD matrix.
 Moreover, \eqref{eq:constraint-dof4}  is equivalent to the following constraints on $\mathcal{N}_i$:
	\begin{equation}\label{eq:constraint-dof4-equivalent}
	\begin{aligned}
		\mathcal{N}_{N-3} &\ge - \frac{(N-2) (N-1) (2 N+1)}{N(2N-3)(N+1)}, \\
		\mathcal{N}_{N-1} &\ge - \frac{N}{N+1} + \frac{h(\mathcal{N}_{N-3},\mathcal{N}_{N-2},\mathcal{N}_{N};N)}{(N-2) (\mathcal{N}_{N-3} (N+1) (2 N-3) N + (N-2) (N-1) (2 N+1))^2}
	\end{aligned}
	\end{equation}
	with $h=h(\mathcal{N}_{N-3},\mathcal{N}_{N-2},\mathcal{N}_{N};N)$ is a function given by
	\begin{equation}\label{eq:constraint-numerator}
	\begin{aligned}
		h ={}& \mathcal{N}_{N-3}^3 (N-1) N^2 (-2 N^2+N+3)^2  \\
		&+ \mathcal{N}_{N-2} (N+1) (2 N-1) (2 N+1) (N-2)^3 (\mathcal{N}_{N-2} N-\mathcal{N}_N (N-1)) \\
		& +\mathcal{N}_{N-3} (N-2)^2 (\mathcal{N}_N (N+1)^2 (4 N^2-8N+3) (\mathcal{N}_{N-2} N - \mathcal{N}_N (N-1))+(2 N+1)^2 (N-1)^3) \\
		& +2 \mathcal{N}_{N-3}^2 (N-1)^2 N (N+1) (2 N-3) (2 N+1) (N-2).
	\end{aligned}
	\end{equation}
\end{thm}
 \begin{proof}
 	The proof is given in Appendix \ref{sec:appendix-proof}.
 \end{proof}

\begin{cor}
	If we keep only 3 degrees of freedoms by setting $a_{N-3}=0$, the constraint \eqref{eq:constraint-dof4} reduces to
	\begin{equation}\label{eq:constraint-dof3}
		a_{N-1} > \frac{2N-1}{(N-1)^2} a_{N-2} \brac{N a_{N-2} - (N-1) a_N},
	\end{equation}
	and the equivalent constraint \eqref{eq:constraint-dof4-equivalent} reduces to
	\begin{equation}
	    \mathcal{N}_{N-1} > -\frac{N}{N+1} + \frac{(2N-1)(N+1)}{(N-1)^2(2N+1)} \mathcal{N}_{N-2} (N \mathcal{N}_{N-2} - (N-1) \mathcal{N}_N).
	\end{equation}
	If we further set $a_{N-3}=a_{N-2}=0$ by taking 2 degrees of freedoms, the constraint \eqref{eq:constraint-dof4} reduces to
	\begin{equation}
		a_{N-1} > 0,
	\end{equation}
	and the equivalent constraint \eqref{eq:constraint-dof4-equivalent} reduces to
	\begin{equation}
		\mathcal{N}_{N-1} > -\frac{N}{N+1}.
	\end{equation}	
	In this case, the coefficient matrix $A$ is a tridiagonal matrix with positive off-diagonal entries, which is real diagonalizable.
\end{cor}

\begin{rem}
	In principle, one could generalize the result in Theorem \ref{thm:hyperbolic} to more than 4  degrees of freedom, by following the same lines in the proof given in Appendix \ref{sec:appendix-proof}. However, the constraints for  hyperbolicity will be a set of implicit inequalities. How to incorporate the implicit constraint into the architecture of the neural network would be an interesting topic to explore.
\end{rem}

Besides the hyperbolicity for the convection term, one may also be interested in the relaxation effect of the source term. In this regard, the related property is the structural stability condition proposed in \cite{yong1999singular}, which includes the constraints on the convection term, collision term, and the coupling of both. This stability condition is shown to be critical for the existence of the solutions \cite{peng2016uniform}, and satisfied by many moment closure systems \cite{di2017linear}. In addition, the structural stability condition to the moment systems is same as H-theorem to the Boltzmann equation, which characterizes the dissipation property of the moment systems. For the convenience of the reader, we review the structural stability condition in Appendix \ref{sec:appendix-structure-stability}. It is easy to show that our new ML moment closure system also satisfied the structural stability condition:
\begin{thm}\label{thm:structure-stability}
	For $k=3$ and $N\ge 3$, if the constraint \eqref{eq:constraint-dof4} (or equivalently \eqref{eq:constraint-dof4-equivalent}) is satisfied, then the ML moment closure system with the closure relation \eqref{eq:anstaz-partial-learn-gradient} satisfies the structural stability condition in \cite{yong1999singular}.
\end{thm}
\begin{proof}
    We only prove for the case $\sigma_a=0$. The other case can be proven by following the same line.
    
    Denote $Q=Q(U)$ the source term of the closure system and $Q_U=\frac{\partial Q}{\partial U}$. Then we can compute
    \begin{equation}
    \begin{aligned}
        A_0 Q_U &= \diag(D, B)\diag(0, -\sigma_s I_N) \\
        &= \diag(D, B)\diag(-C, -\sigma_s  I_3) \\
        &= - \diag(D C, \sigma_s B)
    \end{aligned}
    \end{equation}
    Here, $C=\diag(0, \sigma_s I_{N-3} )$ and $I_r$ denote the identity matrix of order $r$.
    Notice that $A_0 Q_U$ is a symmetric negative semi-definite matrix. Then the structural stability condition can be easily checked for this moment closure system by taking $P$ to be a scalar matrix in Appendix \ref{sec:appendix-structure-stability}.
\end{proof}

\subsection{Diffusion limit}\label{sec:diffusion-limit}

In this part, we show formally that the hyperbolic ML moment closure model has the correct diffusion limit.  
Consider the RTE under a diffusive scaling:
\begin{equation}\label{eq:rte-diffusive-scaling}
	\varepsilon \partial_t f + v \partial_x f = \frac{\sigma_s}{\varepsilon}\brac{\frac{1}{2} \int_{-1}^1 f dv - f} - \varepsilon \sigma_a f,
\end{equation}
where $\varepsilon>0$ denotes the Knudsen number, which is the ratio of mean free path of particles to the characteristic length. It is well known that as $\varepsilon\rightarrow 0$, the kinetic transport equation \eqref{eq:rte-diffusive-scaling} will converge to the diffusion limit---macroscopic model:
\begin{equation}\label{eq:rte-diffusive-limit}
    \partial m_0 = \partial_x\brac{ \frac{1}{3\sigma_s} \partial_x m_0} - \sigma_a m_0,
\end{equation}
where $m_0$ is the zeroth-order moment of $f$.

For the RTE under a diffusive scaling \eqref{eq:rte-diffusive-scaling}, the corresponding moment closure equation is:
\begin{equation}\label{eq:moment-equation-diffusive}
\begin{aligned}
	 \varepsilon \partial_t m_0 + \partial_x m_1 &= - \varepsilon \sigma_a m_0 \\
     \varepsilon \partial_t m_{1} + \frac{1}{3} \partial_x m_{0} + \frac{2}{3} \partial_x m_{2}  &= - ( \frac{\sigma_s}{\varepsilon} + \varepsilon \sigma_a) m_{1}, \\
     &\cdots \\
	 \varepsilon \partial_t m_{N-1} + \frac{N-1}{2N-1} \partial_x m_{N-2} + \frac{N}{2N-1} \partial_x m_{N}  &= - ( \frac{\sigma_s}{\varepsilon} + \varepsilon \sigma_a) m_{N-1}, \\
	 \varepsilon \partial_t m_N + \frac{N}{2N+1} \partial_x m_{N-1} + \frac{N+1}{2N+1} \brac{\sum_{k=0}^N \mathcal{N}_k \partial_x m_k} &= - ( \frac{\sigma_s}{\varepsilon} + \varepsilon \sigma_a) m_N.
\end{aligned}
\end{equation}

Next, we formally show that the ML moment closure model \eqref{eq:moment-equation-diffusive} with $N\ge3$ converges to the diffusion limit \eqref{eq:rte-diffusive-limit} as $\varepsilon\rightarrow0$. We expand $m_k$ for $k=0,\cdots,2$ as:
\begin{equation}
    m_k \sim m_k^{(0)} + \varepsilon m_k^{(1)}  + \varepsilon^2 m_k^{(2)} + \mathcal{O}(\varepsilon^3)
\end{equation}
and plug into the first three equations of \eqref{eq:moment-equation-diffusive}. Collecting the $\mathcal{O}({\varepsilon})$ term of the first equation, we obtain
\begin{equation}
    \partial_t m_0^{(0)} + \partial_x m_1^{(1)} = - \sigma_a m_0^{(0)}.
\end{equation}
Collecting the $\mathcal{O}(1)$ term of the second equation, we have
\begin{equation}
    \frac{1}{3}\partial_x m_0^{(0)} + \frac{2}{3}\partial_x m_2^{(0)} = - \sigma_s m_1^{(1)}.
\end{equation}
Collecting the $\mathcal{O}({\varepsilon}^{-1})$ term of the third equation, we get
\begin{equation}
    m_2^{(0)} = 0.
\end{equation}
Combining the above three relations, we eventually reach:
\begin{equation}
    \partial m_0^{(0)} = \partial_x\brac{ \frac{1}{3\sigma_s} \partial_x m_0^{(0)}} - \sigma_a m_0^{(0)}.
\end{equation}
This indicates that the leading-order term of $m_0$ satisfies the diffusion equation \eqref{eq:rte-diffusive-limit}.

To justify the above formal asymptotic analysis, it turns out that the structural stability condition is essential.
Indeed, under the structural stability condition, one can rigorously prove that the solution of the ML moment closure model will converge to that of the macroscopic model \eqref{eq:rte-diffusive-limit} as $\varepsilon$ goes to 0 \cite{peng2016uniform,lattanzio2001hyperbolic}, if the constraint \eqref{eq:constraint-dof4} (or equivalently \eqref{eq:constraint-dof4-equivalent}) is satisfied and the neural network in \eqref{eq:moment-equation-diffusive} satisfies some regularity conditions. Although the regularity of the neural network seems not easy to be validated, we will show numerically in Section \ref{sec:numerical-test} that our hyperbolic ML closure model can capture the correct diffusion limit.

\section{Training of the neural network}\label{sec:training}

In this section, we provide details of the training of the proposed neural network
that enforces the constraint \eqref{eq:constraint-dof4} (or equivalently \eqref{eq:constraint-dof4-equivalent}) for the hyperbolicity for the case when $k=3$. We design a fully-connected neural network denoted by $\mathcal{M}:\mathbb{R}^{N+1}\rightarrow\mathbb{R}^4$ with the input being the lower order moments $(m_0, m_1, \cdots,m_{N})$ and the output being $\mathcal{M}=(\mathcal{M}_1, \mathcal{M}_2, \mathcal{M}_3, \mathcal{M}_4)$. Then, we do the following post-processing to the output of the neural network:
\begin{equation*}
\begin{aligned}
	& \mathcal{N}_{N} = \mathcal{M}_{4}, \quad \mathcal{N}_{N-2} = \mathcal{M}_{2}, \\
	& \mathcal{N}_{N-3} = \sigma(\mathcal{M}_{3}) - \frac{(N-2) (N-1) (2 N+1)}{N(2N-3)(N+1)} \\
	& \mathcal{N}_{N-1} = \sigma(\mathcal{M}_{1}) - \frac{N}{N+1} + \frac{h(\mathcal{N}_{N-3},\mathcal{N}_{N-2},\mathcal{N}_{N};N)}{(N-2) (\mathcal{N}_{N-3} (N+1) (2 N-3) N + (N-2) (N-1) (2 N+1))^2}
\end{aligned}
\end{equation*}
where $h$ is the function defined in \eqref{eq:constraint-numerator}.
Here $\sigma:\mathbb{R}\rightarrow\mathbb{R}$ is a positive function, i.e., $\sigma(x)>0$ for any $x\in\mathbb{R}$. We test several positive functions to enforce this constraint, including exponential function, ReLU (Rectified Linear Unit) function, softplus function and square function. Our numerical tests show that the neural network with $\sigma$ to be the softplus function has the smallest $L^2$ error in the training data.

In the numerical implementation, we take the number of layers to be 6 and the number of nodes to be 256 in the fully-connected neural network and use the hyperbolic tangent activation function. Other hyperparameters in the training are the same as those in \cite{huang2021gradient}.
The training data comes from numerically solving the RTE using the space-time discontinuous Galerkin (DG) method \cite{crockatt2017arbitrary,crockatt2019hybrid} with a range of initial conditions in the form of truncated Fourier series and different scattering and absorption coefficients which are constants over the computational domain. We train the neural network with 100 different initial data. For each initial data, we run the numerical solver up to $t=1$. The other parameters are the same as in \cite{huang2021gradient}.

\section{Numerical results}\label{sec:numerical-test}

In this section, we show the performance of our ML closure model on a variety of benchmark tests, including problems with constant scattering and absorption coefficients, Gaussian source problems and two-material problems. The main focus of the tests is on the comparison of hyperbolic ML closure (termed as ``hyperbolic closure'') with $k=3$, ML closure with learning gradient in \cite{huang2021gradient} (termed as ``non-hyperbolic closure'') with $k=N$, and the classical $P_N$ closure \cite{chandrasekhar1944radiative}.

In all  numerical examples, we take the physical domain to be the unit interval $[0,1]$ and periodic boundary conditions are imposed.
To numerically solve the moment closure system, we apply the fifth-order finite difference WENO scheme \cite{jiang1996efficient} with a Lax–Friedrichs flux splitting for the spatial discretization, and the third order strong-stability-preserving Runge-Kutta (RK) scheme \cite{shu1988efficient} for the time discretization. We take the grid number in space to be $N_x = 256$. For the hyperbolic ML closure model, the CFL condition is taken to be $\Delta t = 0.8\Delta x/c$ where $c$ denote the maximum eigenvalues in all the grid points, and the penalty constant in the Lax-Friedrichs numerical flux $\alpha_{\textrm{LF}}=c$. For the non-hyperbolic ML closure model, we impose larger numerical diffusion by taking the penalty constant in the Lax-Friedrichs numerical flux $\alpha_{\textrm{LF}}=5$ and fixing $\Delta t = 0.1\Delta x$ \cite{huang2021gradient}.

\begin{exam}[constant scattering and absorption coefficients]\label{ex:const}
	The setup of this example is the same as the data preparation. The scattering and absorption coefficients are taken to be constants over the domain. The initial condition is taken to be a truncated Fourier series. 

	In Figure \ref{fig:const-profile-compare-N6}, we show the numerical solutions of $m_0$ and $m_1$ with seven moments in the closure system ($N=6$). It is observed that, at $t=0.5$, the two ML moment closures agree well with the RTE, except that some minor oscillations appear in the non-hyperbolic ML closure. At $t=1$, the solutions of non-hyperbolic ML moment closure blow up, while the hyperbolic one stays stable and has good agreement with the RTE. As a comparison, the $P_N$ closure has large deviations from the exact solution at both $t=0.5$ and $t=1$.
	\begin{figure}
	    \centering
	    \subfigure[$m_0$ at $t=0.5$]{
	    \begin{minipage}[b]{0.46\textwidth}
	    \includegraphics[width=1\textwidth]{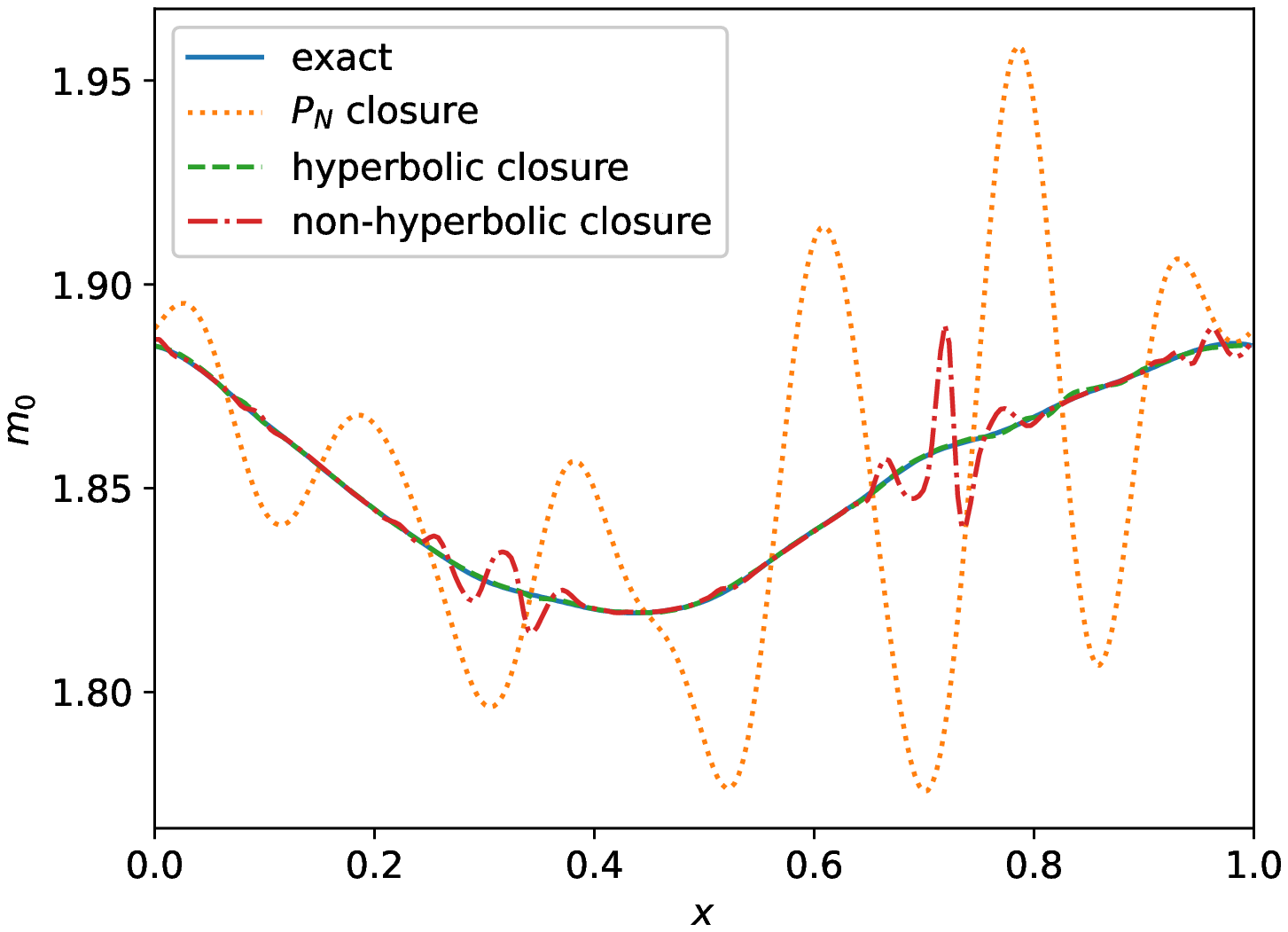}
	    \end{minipage}
	    }
	    \subfigure[$m_1$ at $t=0.5$]{
	    \begin{minipage}[b]{0.46\textwidth}    
	    \includegraphics[width=1\textwidth]{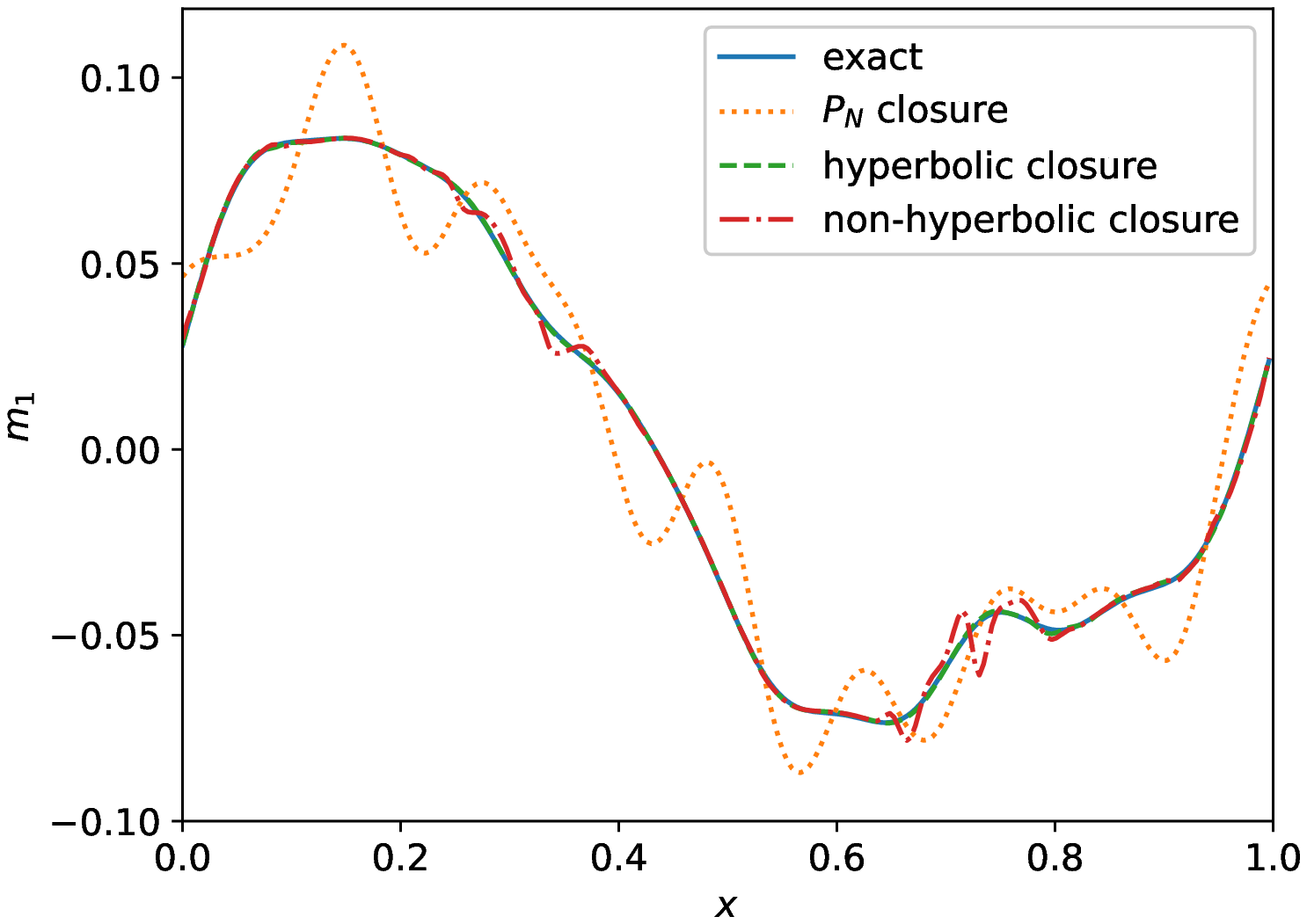}
	    \end{minipage}
	    }
	    \bigskip
	    \subfigure[$m_0$ at $t=1$]{
	    \begin{minipage}[b]{0.46\textwidth}
	    \includegraphics[width=1\textwidth]{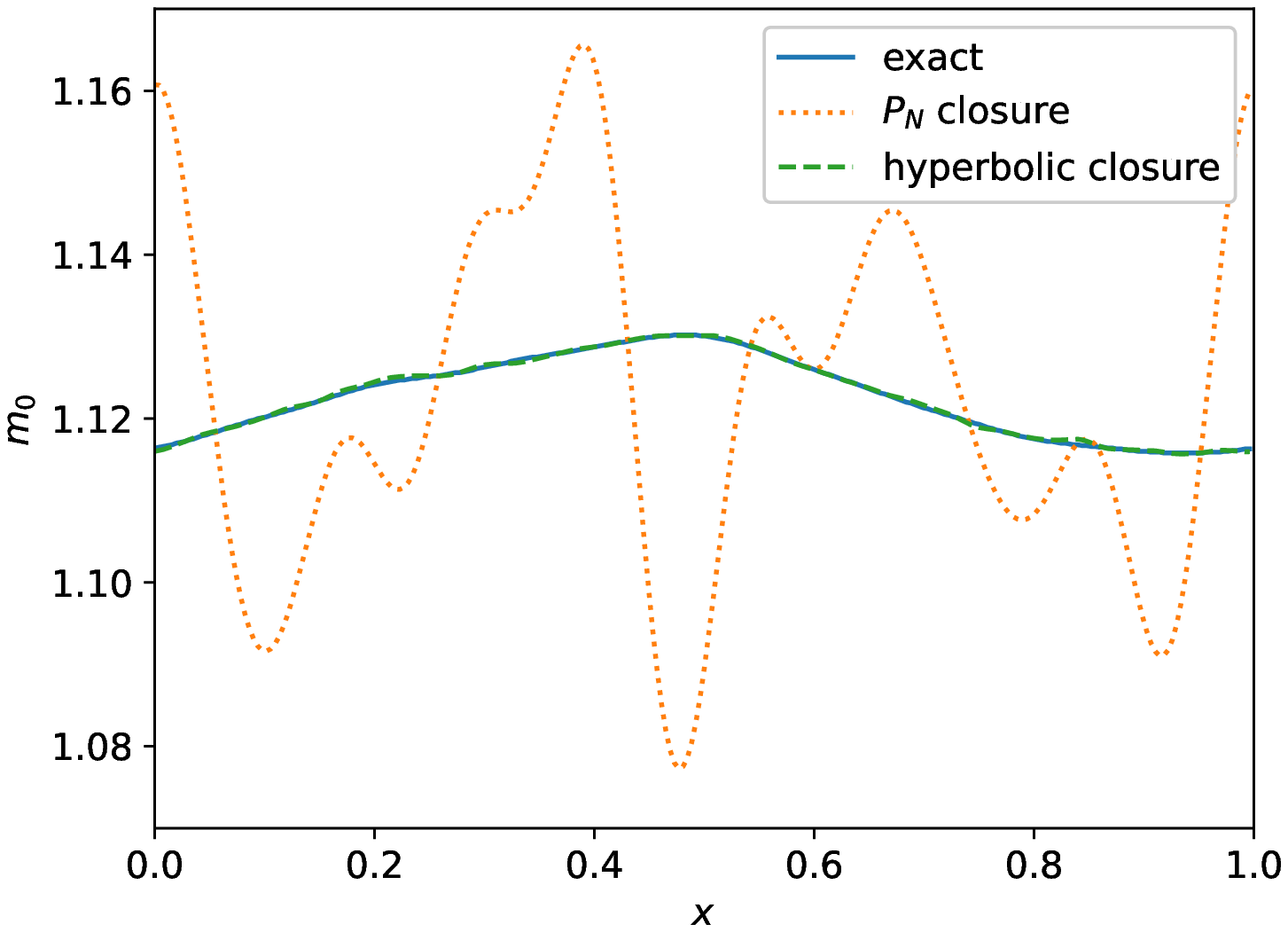}
	    \end{minipage}
	    }
	    \subfigure[$m_1$ at $t=1$]{
	    \begin{minipage}[b]{0.46\textwidth}    
	    \includegraphics[width=1\textwidth]{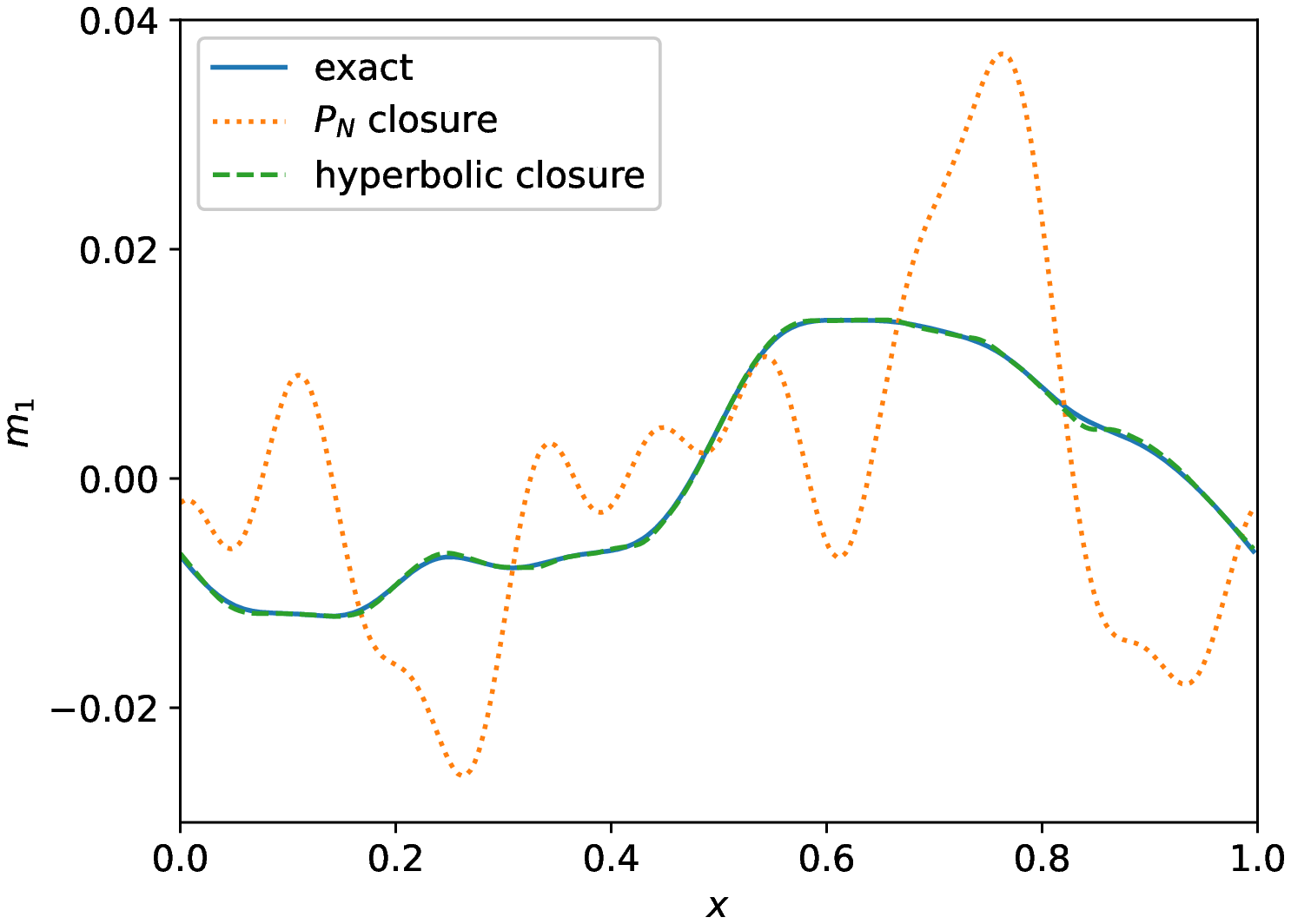}
	    \end{minipage}
	    }    \caption{Example \ref{ex:const}: constant scattering and absorption coefficients, $\sigma_s=\sigma_a=1$, $N=6$, $t=0.5$ and $t=1$. The numerical solution of the non-hyperbolic ML closure blows up at $t=1$.}
	    \label{fig:const-profile-compare-N6}
	\end{figure}

	In Figure \ref{fig:const-loglog-N6}, we display the log–log scatter plots of the relative $L^2$ error versus the scattering coefficient for $N = 6$. We observe that, at $t=0.5$, both ML closures have better accuracy than the $P_N$ closure. Moreover, in the optically thick regime, all the closures perform well. At $t=1$, the numerical solutions of the non-hyperbolic ML moment closure blow up at some data points, while the hyperbolic one still has good accuracy.
	\begin{figure}
	    \centering
	    \subfigure[$m_0$ at $t=0.5$]{
	    \begin{minipage}[b]{0.46\textwidth}
	    \includegraphics[width=1\textwidth]{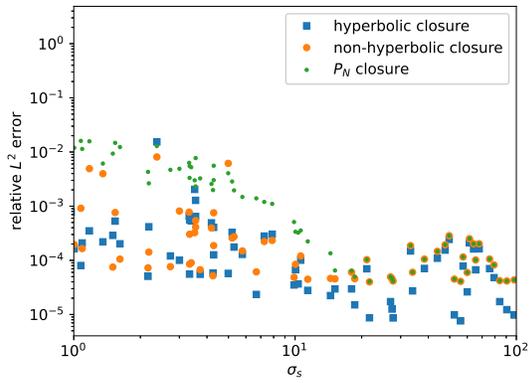}
	    \end{minipage}
	    }
	    \subfigure[$m_1$ at $t=0.5$]{
	    \begin{minipage}[b]{0.46\textwidth}    
	    \includegraphics[width=1\textwidth]{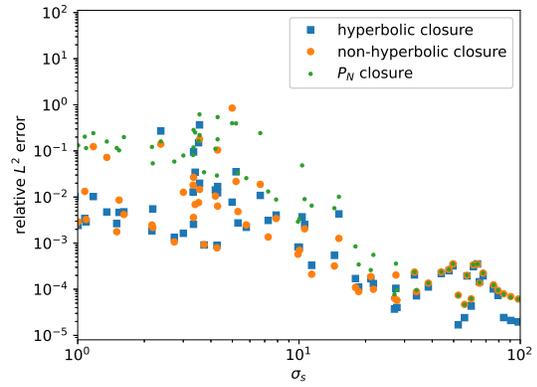}
	    \end{minipage}
	    }
	    \bigskip
	    \subfigure[$m_0$ at $t=1$]{
	    \begin{minipage}[b]{0.46\textwidth}
	    \includegraphics[width=1\textwidth]{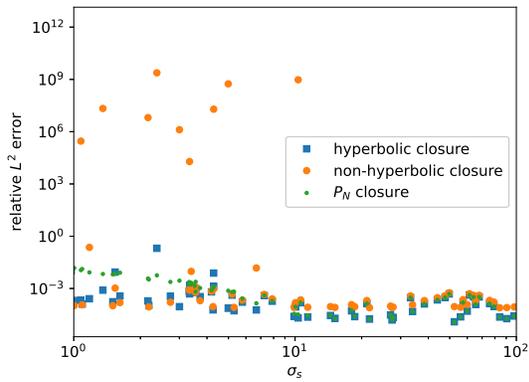}
	    \end{minipage}
	    }
	    \subfigure[$m_1$ at $t=1$]{
	    \begin{minipage}[b]{0.46\textwidth}    
	    \includegraphics[width=1\textwidth]{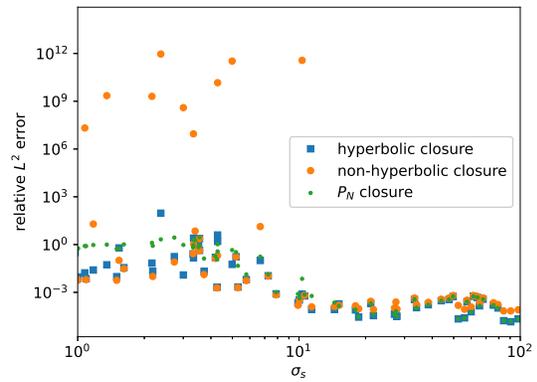}
	    \end{minipage}
	    }    \caption{Example \ref{ex:const}: constant scattering and absorption coefficients, $\sigma_s=\sigma_a=1$, $N=6$, $t=0.5$ and $t=1$. The non-hyperbolic closure blows up at $t=1$.}
	    \label{fig:const-loglog-N6}
	\end{figure}

	In Figure \ref{fig:const-evolution} (a), we present the $L^2$ error between the solutions of the hyperbolic ML moment closure system and the solution generated by the RTE in the optically thin regime ($\sigma_s=\sigma_a=1$). We observe that the hyperbolic closure generates good predictions in the long time simulation up to $t=10$. We also display the maximum eigenvalues in all the grid points during the time evolution in Figure \ref{fig:const-evolution} (b). It is observed that the eigenvalues are always real numbers, which validates the hyperbolicity feature of the closure system. However, the eigenvalues could be as large as four during the time evolution, which violates the physical characteristic speed of 1 in the RTE. This also results in smaller time step size in the simulation of the closure system. It is unclear how to preserve the physical characteristic speed in the moment closure system in this setting. 
	\begin{figure}
	    \centering
	    \subfigure[$L^2$ errors of $m_0$ and $m_1$]{
	    \begin{minipage}[b]{0.46\textwidth}
	    \includegraphics[width=1\textwidth]{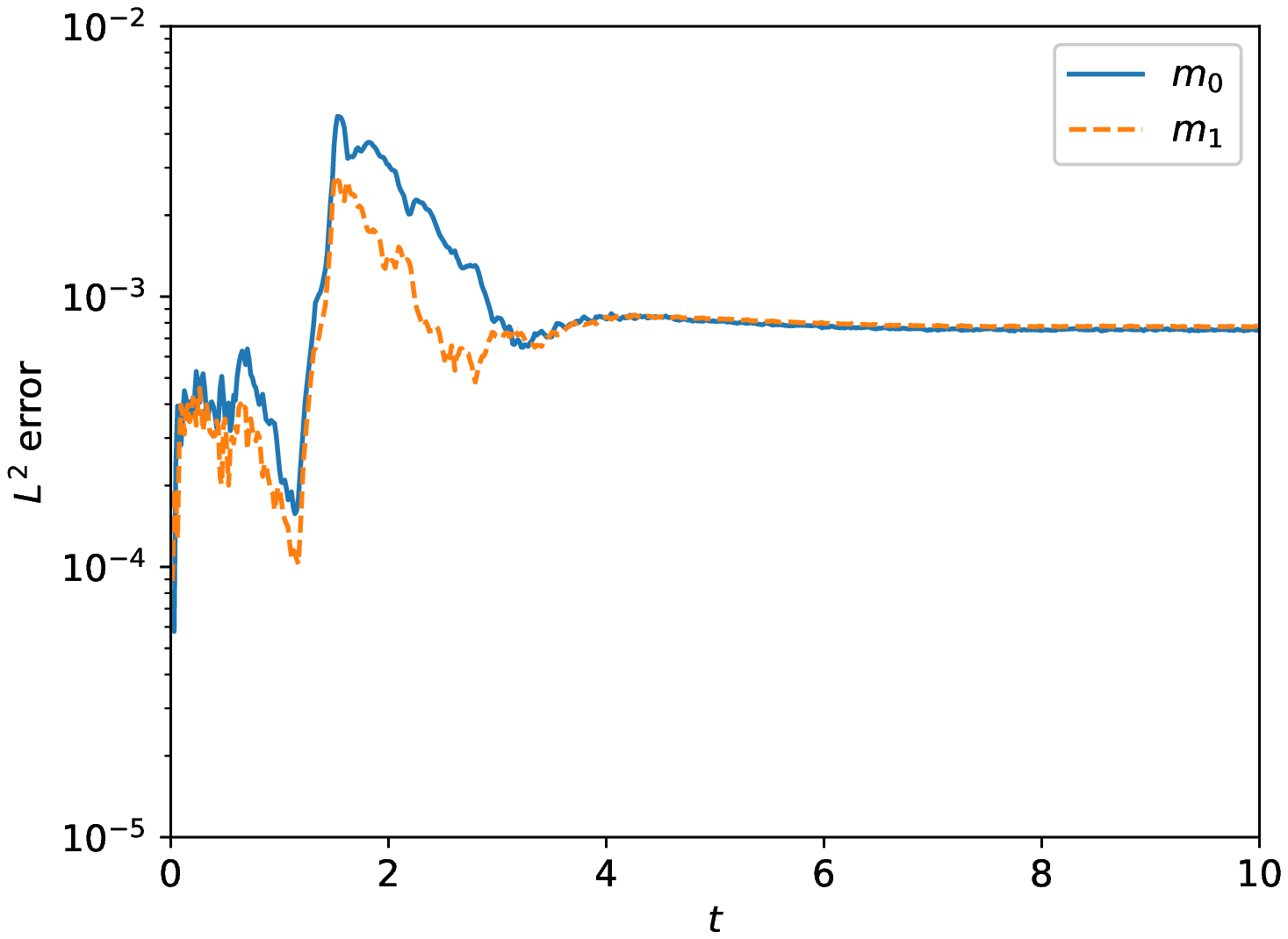}
	    \end{minipage}
	    }
	    \subfigure[maximum eigenvalues]{
	    \begin{minipage}[b]{0.46\textwidth}    
	    \includegraphics[width=1\textwidth]{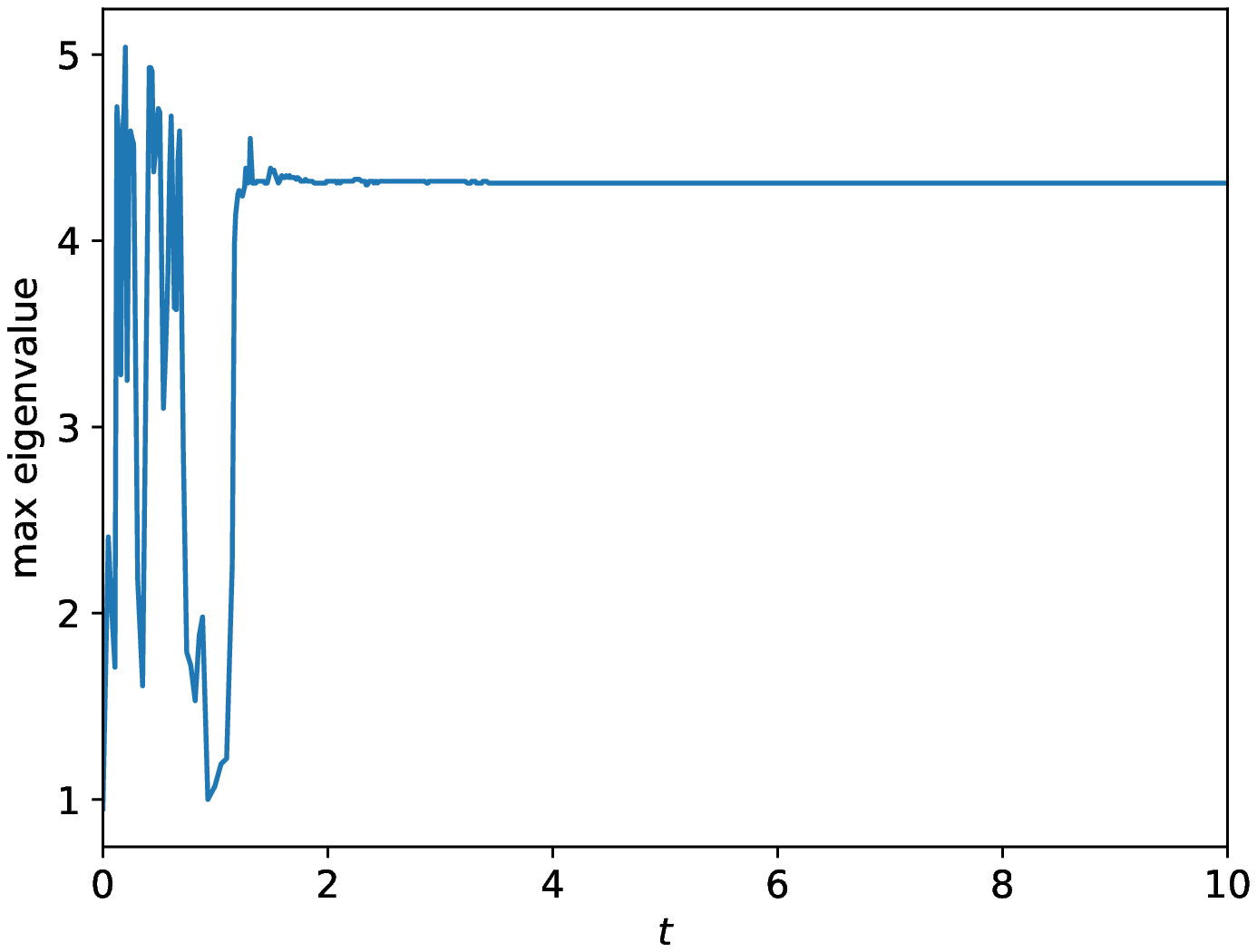}
	    \end{minipage}
	    }
	    \caption{Example \ref{ex:const}: constant scattering and absorption coefficients, $\sigma_s=\sigma_a=1$, $N=6$. Results obtained by the  hyperbolic closure.}
	    \label{fig:const-evolution}
	\end{figure}
\end{exam}

\begin{exam}[Gaussian source problem]\label{ex:gauss-source}
	In this example, we investigate the RTE with the initial condition to be Gaussian distribution in the physical domain, named Gaussian source problem in literature \cite{frank2012perturbed,fan2020nonlinear}:
	\begin{equation}\label{eq:gauss-source-init}
		f_0(x,v) = \frac{c_1}{(2 \pi \theta)^{1/2}} \exp\brac{-\frac{(x - x_0)^2}{2 \theta}} + c_2.
	\end{equation}
	In this test, we take $c_1=0.5$, $c_2=2.5$, $x_0=0.5$ and $\theta=0.01$. 

	In Figure \ref{fig:gauss-source-compare}, we present the results obtained using various closure models. We observe good agreement between the two ML closure model and the kinetic model at $t=1$, while the $P_N$ model has large deviations from the kinetic model. Moreover, the non-hyperbolic closure is more accurate than the hyperbolic closure, perhaps due to the fact that the non-hyperbolic closure has larger $k$ values and thus can approximate the closure better. The solution of the non-hyperbolic ML closure model blows up at $t=2$, while the hyperbolic ML model stays stable and is more accurate than the $P_N$ closure. Although the deviation of the hyperbolic ML model ``looks'' large in Figure \ref{fig:gauss-source-compare} (c), this is due to the fact that the exact solution is close to constant over the computational domain and the true error is actually small. The relative $L^2$ and the relative $L^{\infty}$ error between the solution of the hyperbolic ML model and that of the exact solution at $t=2$ is $1.40\times10^{-3}$ and $3.70\times10^{-3}$, respectively. This again illustrates that the hyperbolicity is an essential property in the ML closure model. 
	\begin{figure}
	    \centering
	    \subfigure[$m_0$ at $t=1$]{
	    \begin{minipage}[b]{0.46\textwidth}
	    \includegraphics[width=1\textwidth]{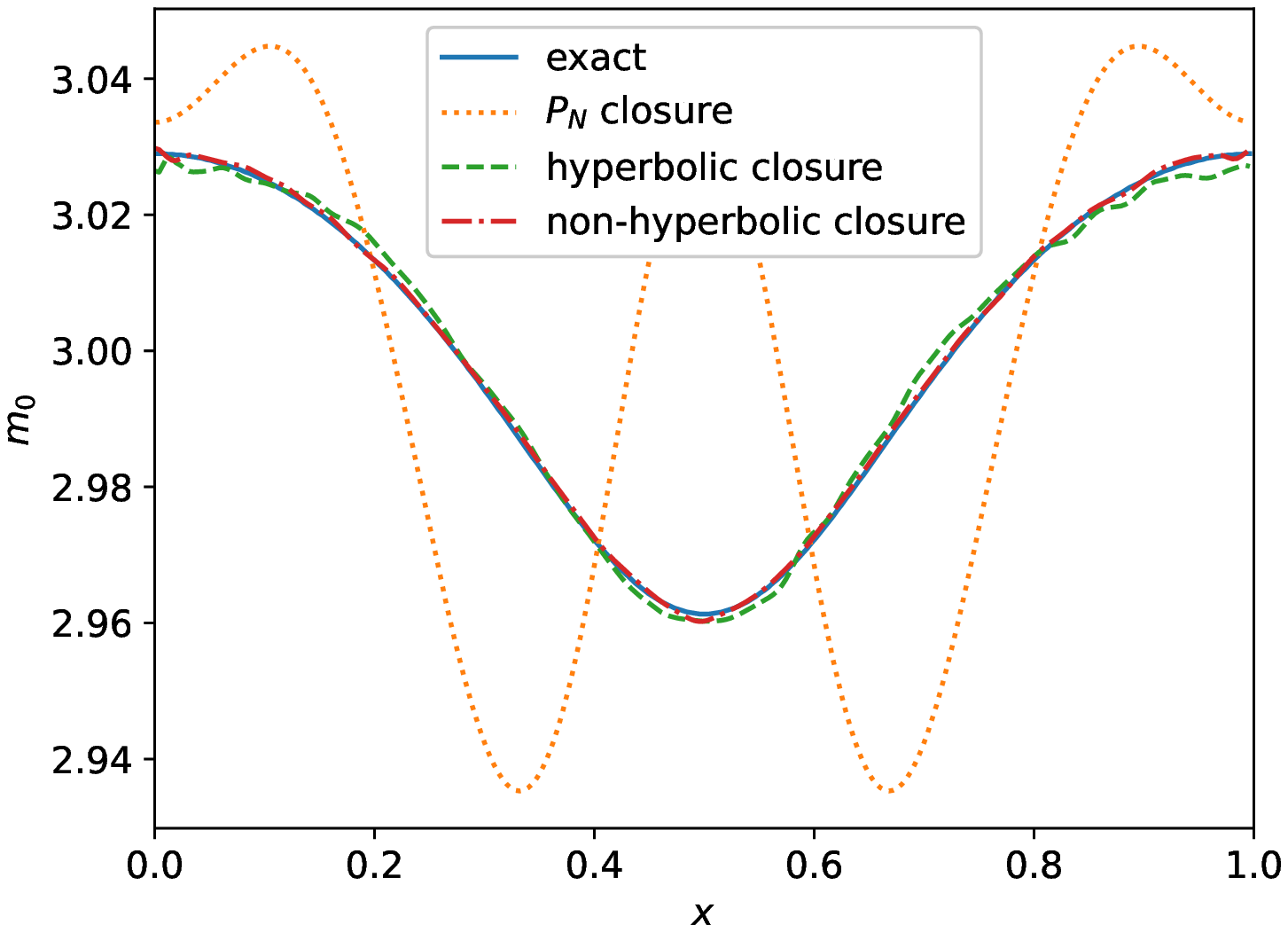}
	    \end{minipage}
	    }
	    \subfigure[$m_1$ at $t=1$]{
	    \begin{minipage}[b]{0.46\textwidth}    
	    \includegraphics[width=1\textwidth]{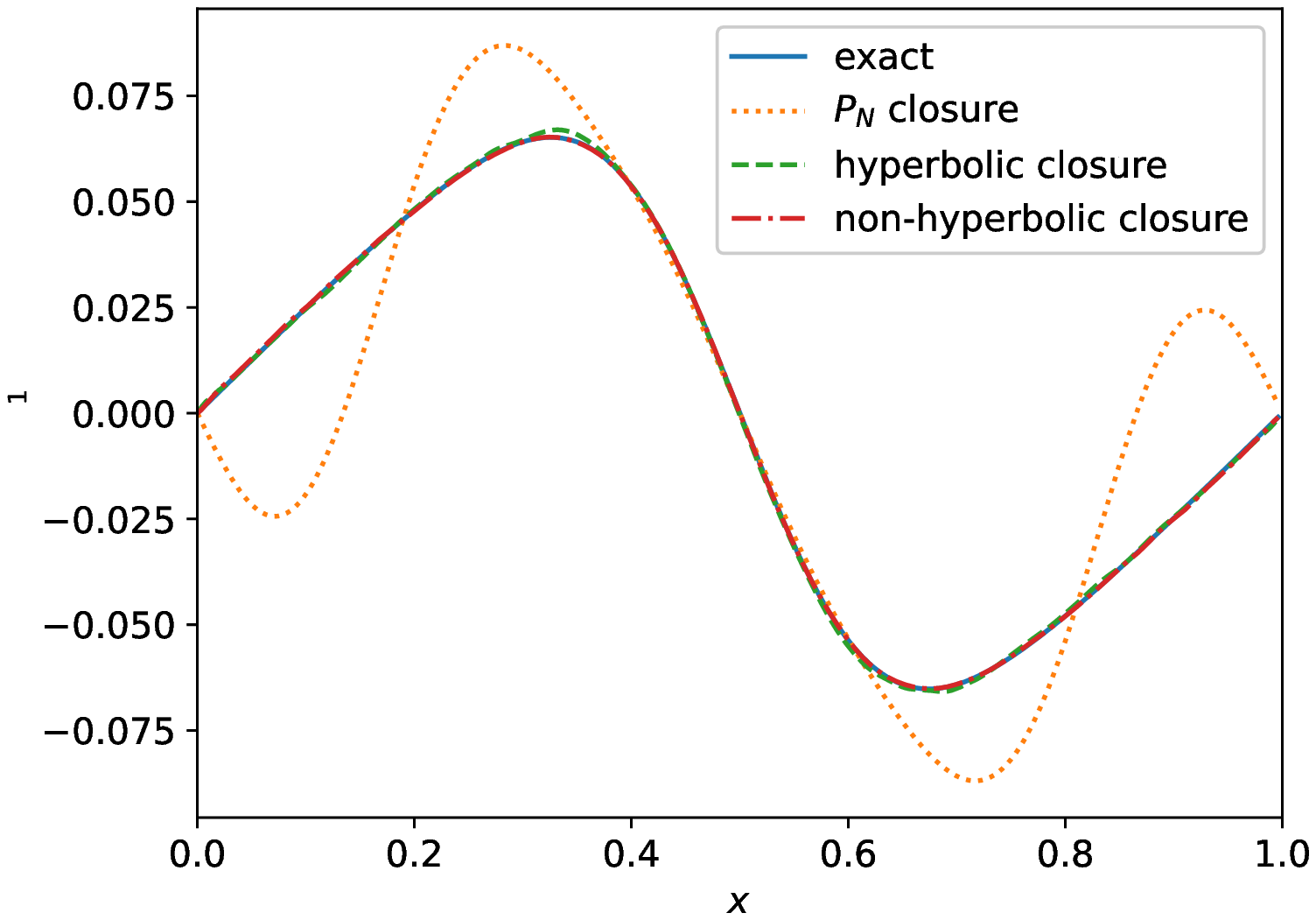}
	    \end{minipage}
	    }
	    \bigskip
	    \subfigure[$m_0$ at $t=2$]{
	    \begin{minipage}[b]{0.46\textwidth}
	    \includegraphics[width=1\textwidth]{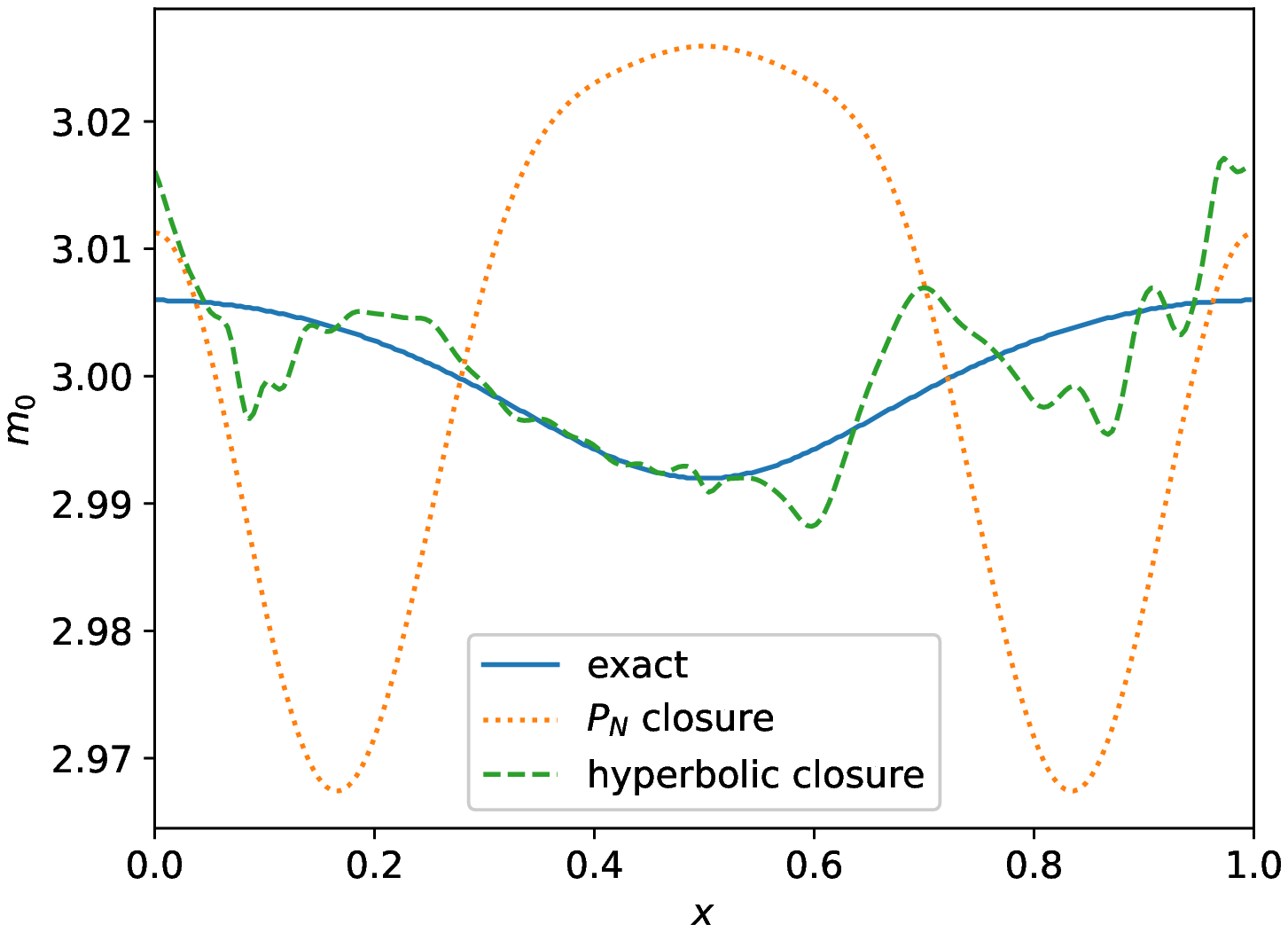}
	    \end{minipage}
	    }
	    \subfigure[$m_1$ at $t=2$]{
	    \begin{minipage}[b]{0.46\textwidth}    
	    \includegraphics[width=1\textwidth]{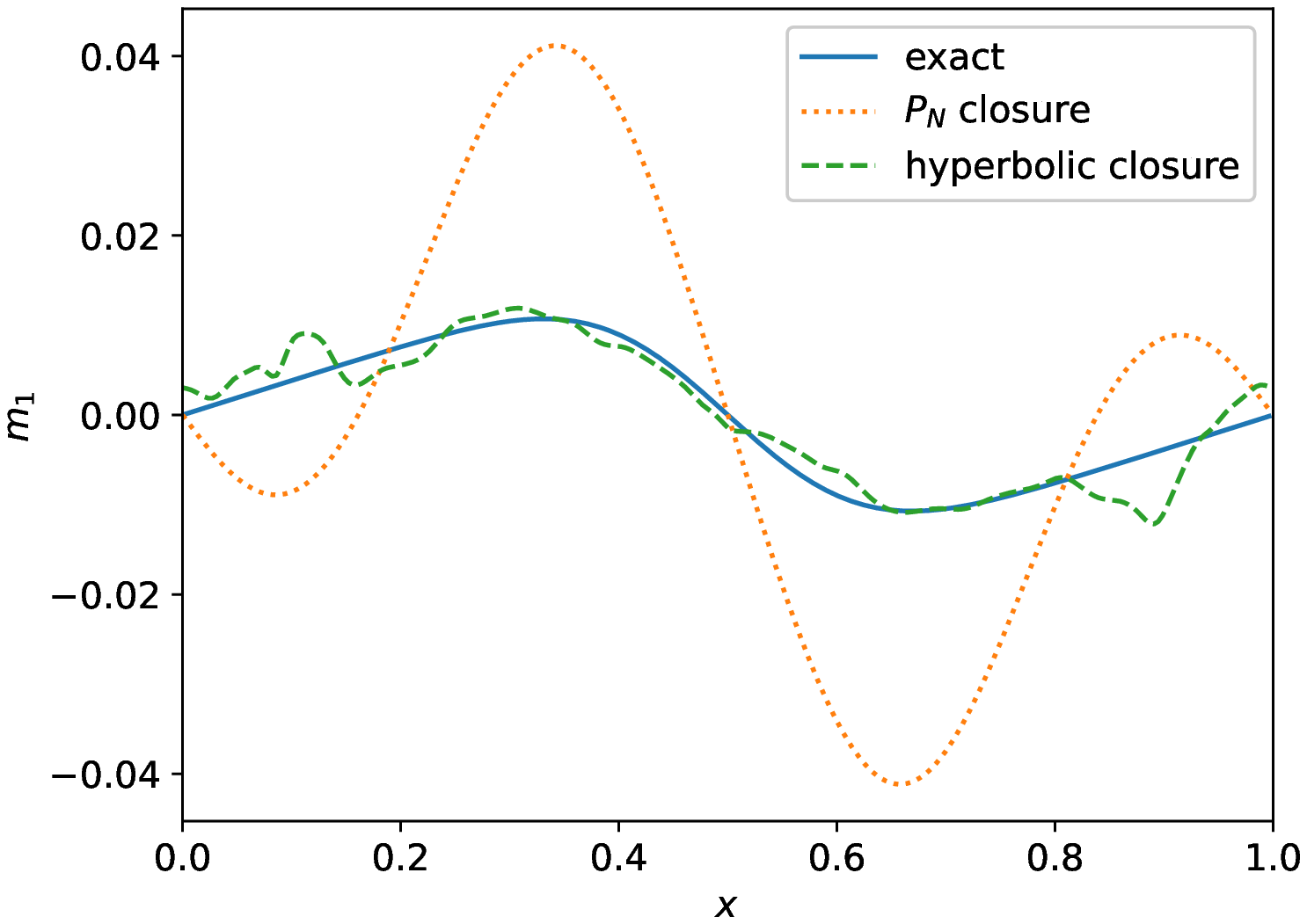}
	    \end{minipage}
	    }    \caption{Example \ref{ex:gauss-source}: Gaussian source problem, $\sigma_s=\sigma_a=1$, $N=6$, $t=1$ and $t=2$. The non-hyperbolic closure blows up at $t=2$.}
	    \label{fig:gauss-source-compare}
	\end{figure}
\end{exam}

\begin{exam}[two-material problem]\label{ex:two-material}
	In this example, we consider the two-material problem \cite{larsen1989asymptotic}. In the problem setup, there exist two discontinuities $0<x_1<x_2<1$ in the domain, and $\sigma_s$ and $\sigma_a$ are piecewise constant functions:
	$$ 
	\sigma_s(x)=\left\{
	\begin{aligned}
	& \sigma_{s1}, \quad  ~ x_1 < x < x_2, \\
	& \sigma_{s2}, \quad  ~ 0\le x < x_1 ~ \textrm{or} ~ x_2\le x < 1.
	\end{aligned}
	\right.
	$$
	and
	$$ 
	\sigma_a(x)=\left\{
	\begin{aligned}
	& \sigma_{a1}, \quad  ~ x_1 < x < x_2, \\
	& \sigma_{a2}, \quad  ~ 0\le x < x_1 ~ \textrm{or} ~ x_2\le x < 1.
	\end{aligned}
	\right.
	$$

	In the numerical example, we take $x_1=0.3$, $x_2=0.7$, $\sigma_{s1}=1$, $\sigma_{s2}=10$ and $\sigma_{a1}=\sigma_{a2}=0$. The numerical results are shown in Figure \ref{fig:two-material-compare}. The gray part is in the optically thin regime and the other part is in the intermediate regime. We observe that the two ML closure models are more accurate than the $P_N$ closure at $t=0.5$ and the hyperbolic ML closure is slightly better than the non-hyperbolic one. At $t=1$, the hyperbolic ML closure still agrees well with the kinetic model. Large deviations exist for the $P_N$ closure in the optically thin regime and severe oscillations occur for the non-hyperbolic ML closure.
	\begin{figure}
	    \centering
	    \subfigure[$m_0$ at $t=0.5$]{
	    \begin{minipage}[b]{0.46\textwidth}
	    \includegraphics[width=1\textwidth]{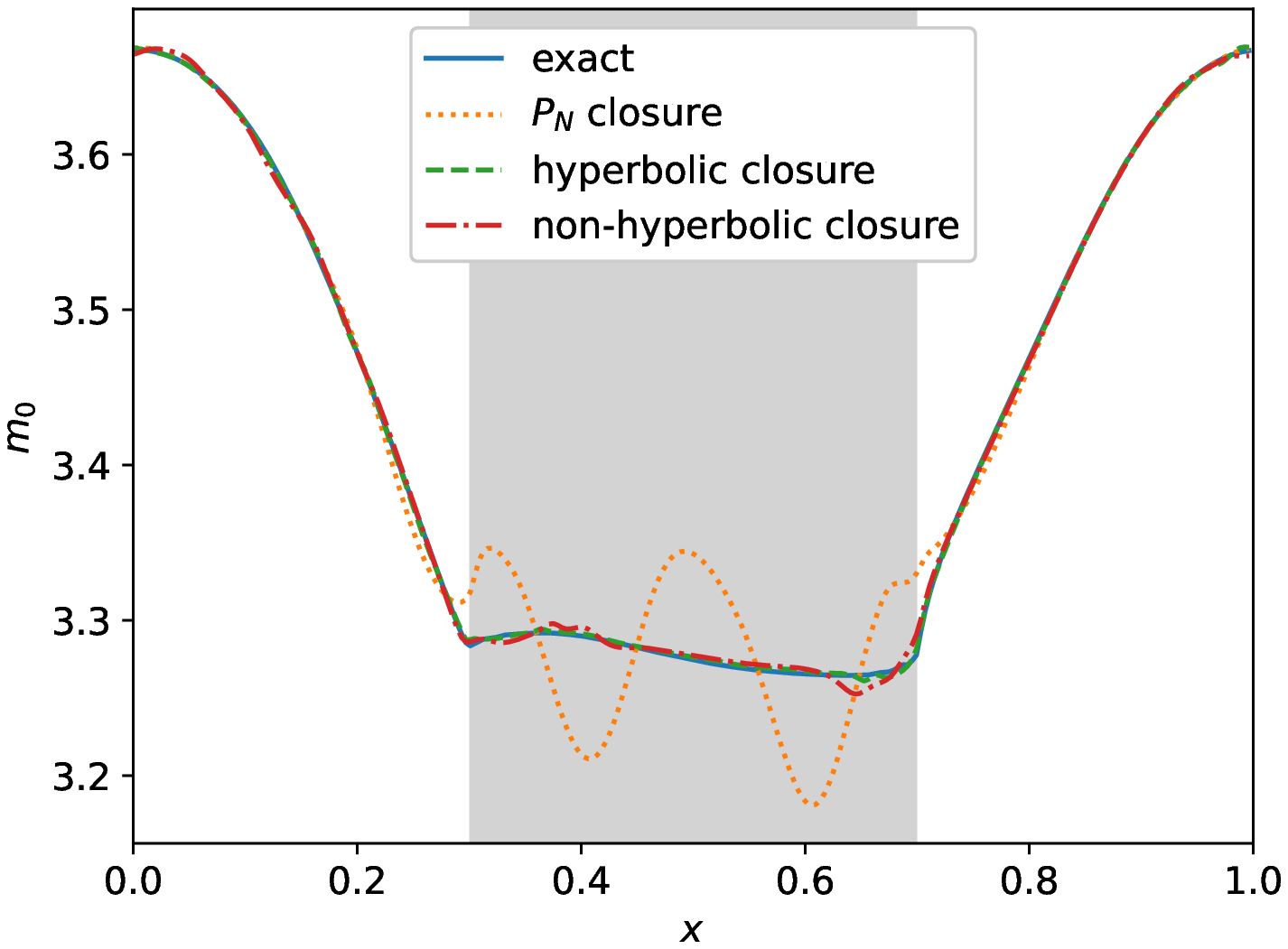}
	    \end{minipage}
	    }
	    \subfigure[$m_1$ at $t=0.5$]{
	    \begin{minipage}[b]{0.46\textwidth}    
	    \includegraphics[width=1\textwidth]{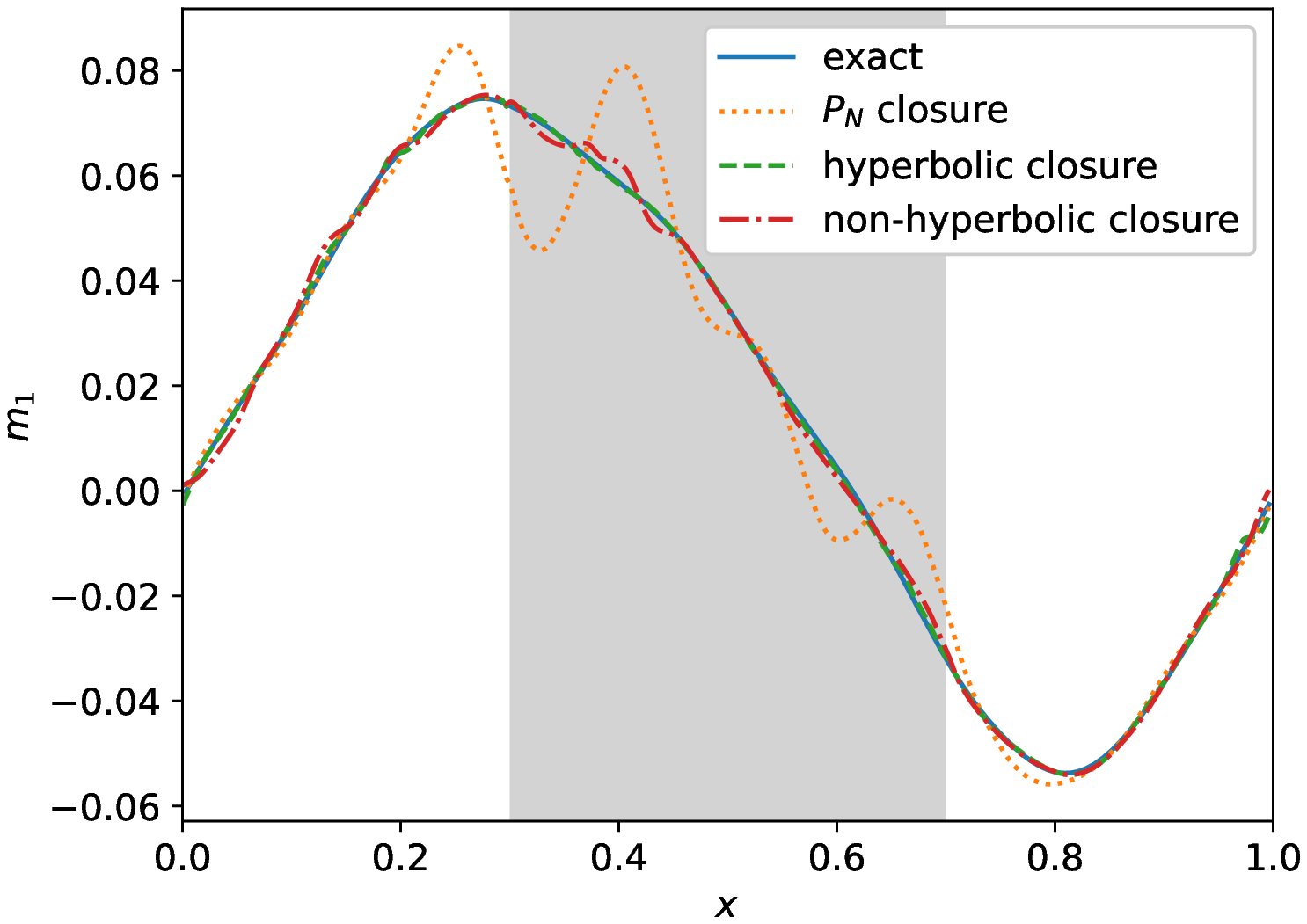}
	    \end{minipage}
	    }
	    \bigskip
	    \subfigure[$m_0$ at $t=1$]{
	    \begin{minipage}[b]{0.46\textwidth}
	    \includegraphics[width=1\textwidth]{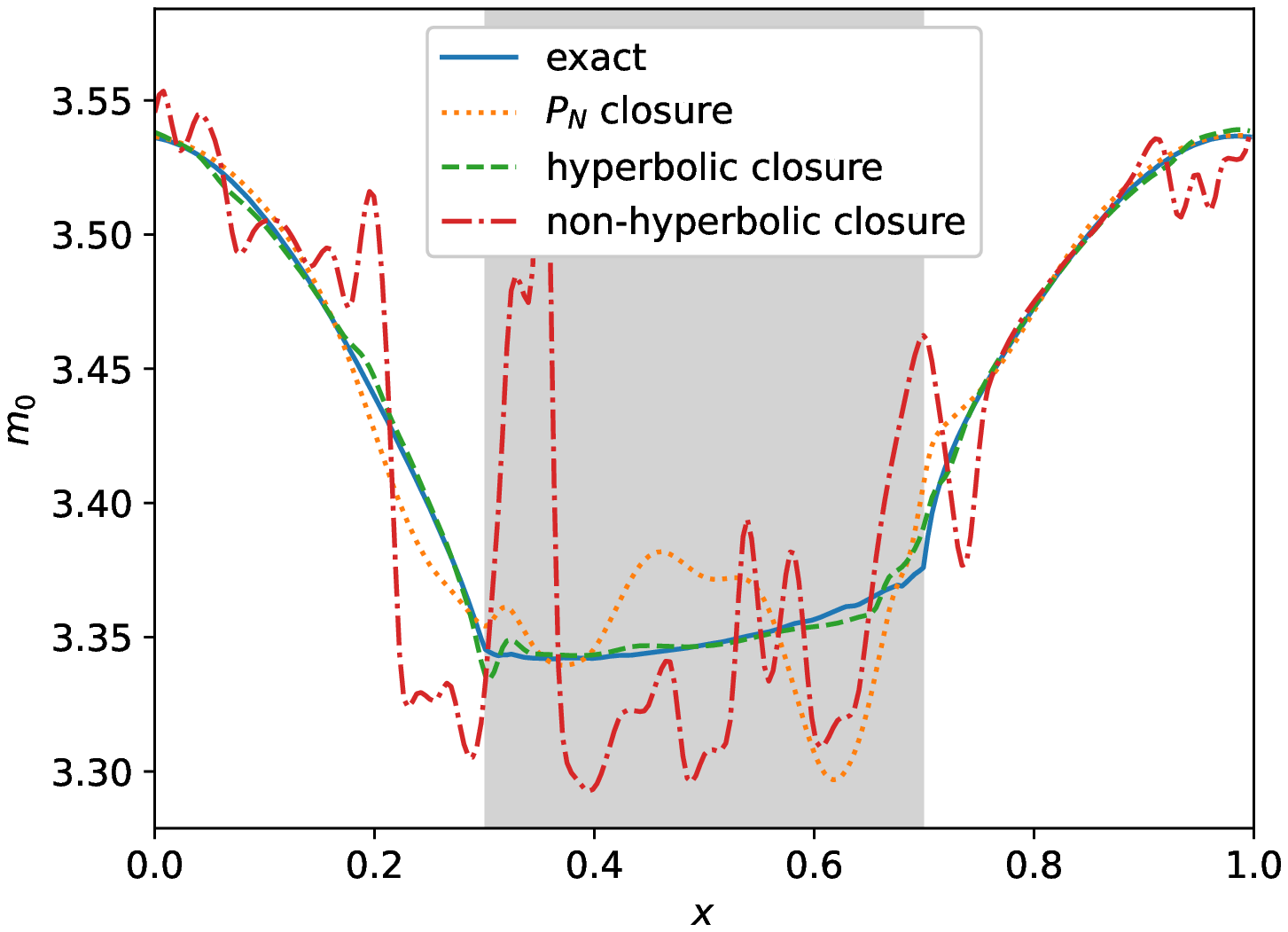}
	    \end{minipage}
	    }
	    \subfigure[$m_1$ at $t=1$]{
	    \begin{minipage}[b]{0.46\textwidth}    
	    \includegraphics[width=1\textwidth]{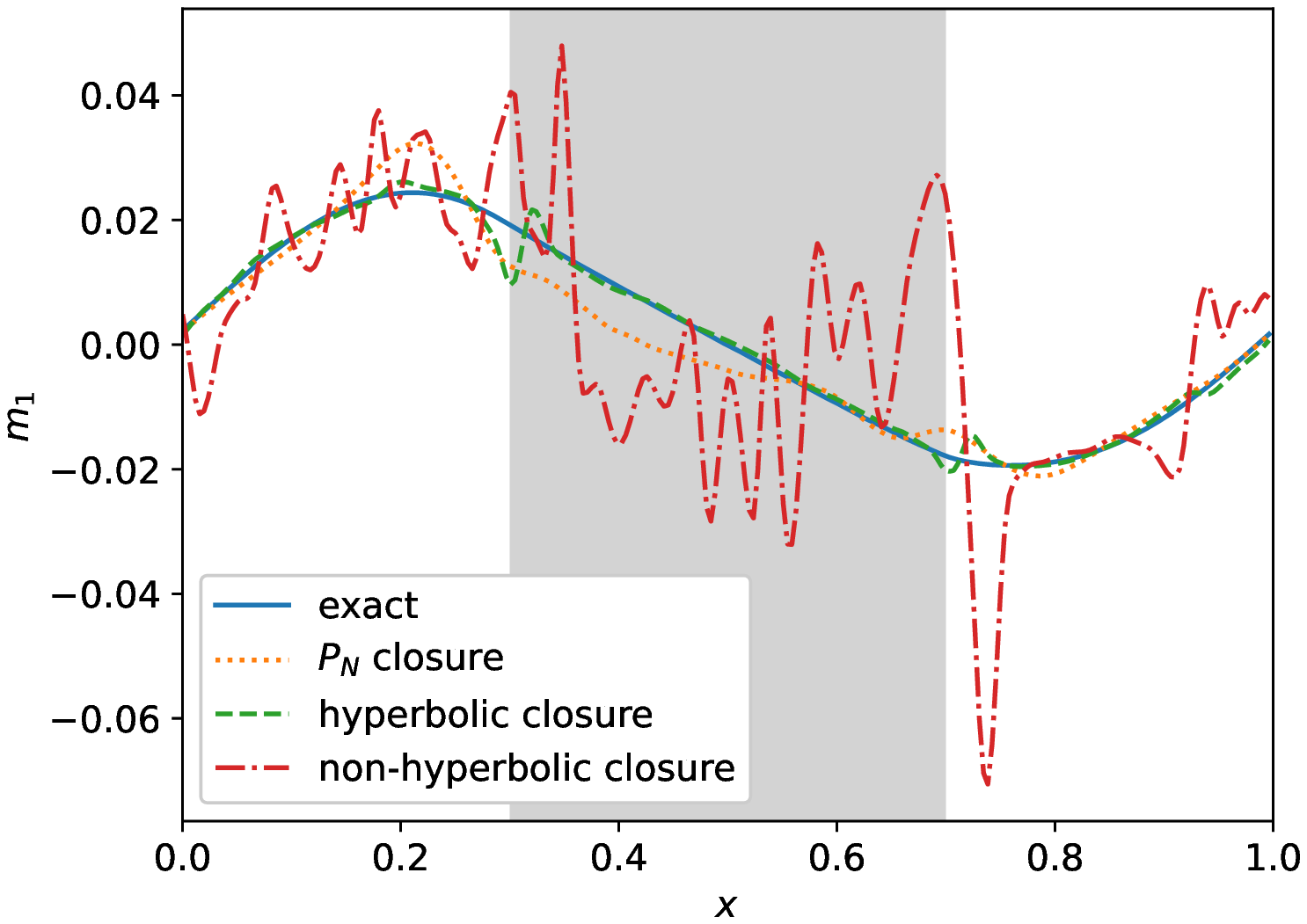}
	    \end{minipage}
	    }    
	    \caption{Example \ref{ex:two-material}:  two material problem. Numerical solutions of $m_0$ and $m_1$ at $t=0.5$ and $t=1$ with $N=6$. The gray part in the middle is in the optically thin regime and the other part is in the intermediate regime.}
	    \label{fig:two-material-compare}
	\end{figure}

\end{exam}

\begin{exam}[diffusion limit]\label{ex:diffusion-limit}
    In the last example, we show that our hyperbolic ML closure model can capture the correct diffusion limit of the RTE. To verify this, we numerically solve the moment closure equation under a diffusive scaling \eqref{eq:moment-equation-diffusive} with the initial condition
    \begin{equation}
    \begin{aligned}
        m_0(x,0) &= \sin(2\pi x) + 2, \\
        m_k(x,0) &= 0, \quad k = 1,\cdots,N
    \end{aligned}
    \end{equation}
    and different values of $\varepsilon$. We take $\sigma_s=1$ and $\sigma_a=0$ on the computational domain. We also numerically solve the diffusion limit equation \eqref{eq:rte-diffusive-limit}.
    In Figure \ref{fig:diffusion-limit}, we show the numerical solutions of the ML moment closure model at $t=0.1$ with $\varepsilon=0.5, 0.1, 0.05, 0.01$. We observe that the numerical solution $m_0$ of the closure models converges to that of the diffusion equation \eqref{eq:rte-diffusive-limit} as $\varepsilon$ approaches zero. This  validates the formal asymptotic analysis in Section \ref{sec:diffusion-limit}.
    \begin{figure}
        \centering
        \includegraphics[width=0.5\textwidth]{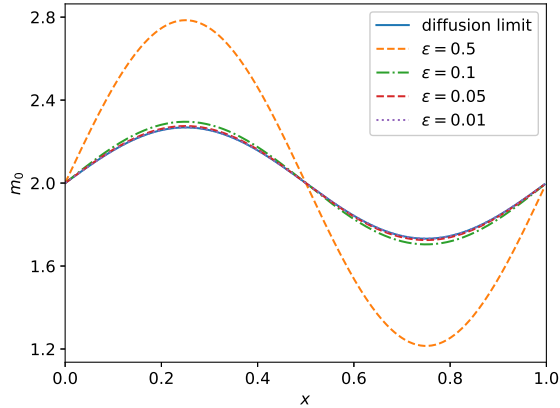}
        \caption{Example \ref{ex:diffusion-limit}: diffusion limit: the solution to the diffusion equation \eqref{eq:rte-diffusive-limit}; other lines: the solutions to the ML moment closure model \eqref{eq:moment-equation-diffusive} with $\varepsilon=0.5, 0.1, 0.05, 0.01$, $t=0.1$.}
        \label{fig:diffusion-limit}
    \end{figure}

\end{exam}

\section{Concluding remarks}\label{sec:conclusion}

In this paper, we propose a method to enforce the global hyperbolicity of the ML closure model. We find a symmetrizer (a symmetric positive definite matrix) for the closure system, and derive constraints that guarantee the system is globally symmetrizable hyperbolic. Moreover, we show that the closure system also inherits the dissipativeness of the RTE by checking the structural stability condition. A variety of benchmark tests including  Gaussian source problem and the two-material problem show good accuracy, correct diffusion limit and generalization ability of our ML closure model. The new approach also offers long time stability by ensuring mathematical consistency between the closure system and the macroscopic model.  Further, the new method demonstrates the plausibility of capturing kinetic effects in a moment system with a handful of moments and an appropriate closure model.

There are several issues that are worthy of further investigations. First, one could in principle generalize the result in Theorem \ref{thm:hyperbolic} to more than 4 degrees of freedom, by following the same lines in the proof given in Appendix \ref{sec:appendix-proof}. In this general case, the hyperbolicity constraints will be a set of implicit inequalities. How to incorporate the constraint of implicit inequalities into the architecture of the neural network would be an interesting topic to explore. Moreover, the moment closure model is expected to have better accuracy with more degrees of freedom. 
Another topic is that the characteristic speed of the current model may exceed the physical bound, as we observed in the numerical tests. This unphysical characteristic speed results in a small time step size in the computation, and is something we would like to address in our future work. Another interesting topic is to extend the current approach to the two dimensional case. These issues constitute the body of our ongoing work.

\section*{Acknowledgment}

We thank Michael M. Crockatt from Sandia National Laboratories for providing a numerical solver for the radiative transfer equations. We also would like to acknowledge the High Performance Computing Center (HPCC) at Michigan State University for providing computational resources that have contributed to the research results reported within this paper.

\begin{appendices}

\section{Hyperbolicity}\label{sec:appendix-hyperbolic}

In this part, we review the definition of the hyperbolicity and some equivalent conditions. See also the details in Chapter 3 in \cite{serre1999systems}.

Consider the first-order system of equations in 1D:
\begin{equation}\label{eq:appendix-def-hyperbolic}
    U_t + A(U) U_{x} = Q(U).
\end{equation}
Here $U=U(x,t)$ is an unknown $n$-vector valued function, $Q(U)$ and $A(U)$ are given $n$-vector and $n\times n$-matrix valued smooth functions of $U\in\mathcal{G}$ (an open subset of $\mathbb{R}^n$ called state space), respectively.
\begin{de}[hyperbolic]
    The system \eqref{eq:appendix-def-hyperbolic} is hyperbolic at $U_0$ if the matrix $ A(U_0)$ is real diagonlizable. The system \eqref{eq:appendix-def-hyperbolic} is  globally hyperbolic if it is hyperbolic at any $U_0\in\mathcal{G}$.
\end{de}
\begin{de}[symmetrizable hyperbolic]
    The system \eqref{eq:appendix-def-hyperbolic} is called symmetrizable hyperbolic if there exists a symmetric positive definite matrix $A_0$ such that $A_0A$ is symmetric.
\end{de}
The following classical conclusion holds for these two definitions:
\begin{prop}
    The symmetrizable hyperbolicity is a necessary and sufficient condition for the first-order system \eqref{eq:appendix-def-hyperbolic} to be hyperbolic.
\end{prop}
\begin{proof}
We start by the proof of sufficiency.
    By the definition of symmetrizable hyperbolic, there exists a symmetric positive definite (SPD) matrix $A_0$ such that $A_0A$ are symmetric. Since $A_0$ is a SPD matrix, there exists an invertible symmetric matrix $B$ satisfying $A_0=B^2$. Then, we compute
    \begin{equation}
        B A B^{-1}
        = B A_0^{-1} A_0 A B^{-1}
        =  B^{-1} (A_0 A) B^{-1}.
    \end{equation}
    Thus, $B A B^{-1}$ is symmetric since $A_0 A$ and $B$ are symmetric. Then we have $A$ is real diagonlizable.
    
    On the other hand, suppose that $A$ is real diagonlizable, there exist an invertible real matrix $P$ such that $A=P^{-1}DP$ with $D$ a real diagnoal matrix. Take $A_0 = P^{T} P$. Then $A_0$ is a SPD matrix and $A_0 A=P^TDP$  is symmetric.
\end{proof}

\section{Proof of Theorem \ref{thm:hyperbolic}}\label{sec:appendix-proof}

\begin{proof}
	To save space, we only present the proof in the case of $a_{N-3}=0$. One also can prove the case of $a_{N-3}\ne 0$ by following the same line.

	We first write $A$ in \eqref{eq:thm-jacobi-matrix} into a block matrix:
	\begin{equation}
		A = 
		\begin{pmatrix}
		A_1 & A_2 \\
		A_3 & A_4
		\end{pmatrix}
	\end{equation}
	with $A_1\in\mathbb{R}^{(N-1)\times(N-1)}$, $A_2\in\mathbb{R}^{(N-1)\times2}$, $A_3\in\mathbb{R}^{2\times(N-1)}$ and $A_4\in\mathbb{R}^{2\times2}$. Then we compute $A_0 A$:
	\begin{equation}
		A_0 A = 
		\begin{pmatrix}
		D & 0_{(N-2)\times 2} \\
		0_{2\times(N-2)} & B
		\end{pmatrix}
		\begin{pmatrix}
		A_1 & A_2 \\
		A_3 & A_4
		\end{pmatrix}
		=
		\begin{pmatrix}
		D A_1 & D A_2 \\
		B A_3 & B A_4
		\end{pmatrix}
	\end{equation}
	Here, $0_{m\times n}$ denote the zero matrix of size $m\times n$.
	It is easy to see that $D A_1$ is symmetric. 
	
	Next, we compute other blocks in $A_0 A$:
	\begin{equation}
		D A_2 = 
		\diag(1,3,5,\cdots,2N-3)
		\begin{pmatrix}
		0_{(N-2)\times 1} & 0_{(N-2)\times 1} \\
		\frac{N-1}{2N-3} & 0
		\end{pmatrix}
		= 
		\begin{pmatrix}
		0_{(N-2)\times 1} & 0_{(N-2)\times 1} \\
		N-1 & 0
		\end{pmatrix}
	\end{equation}	
	Let
	\begin{equation}
		B = 
		\begin{pmatrix}
		b_{11} & b_{12} \\
		b_{12} & b_{22}
		\end{pmatrix}.
	\end{equation}
	Then, we compute
	\begin{equation}
		B A_3 =
		\begin{pmatrix}
		b_{11} & b_{12} \\
		b_{12} & b_{22}
		\end{pmatrix}
		\begin{pmatrix}
		0_{1\times(N-2)} & \frac{N-1}{2N-1} \\
		0_{1\times(N-2)} & a_{N-2}
		\end{pmatrix}
		=
		\begin{pmatrix}
		0_{1\times(N-2)} & b_{11} \frac{N-1}{2N-1} + b_{12} a_{N-2} \\
		0_{1\times(N-2)} & b_{12} \frac{N-1}{2N-1} + b_{22} a_{N-2}
		\end{pmatrix}
	\end{equation}
	and
	\begin{equation}
		B A_4 =
		\begin{pmatrix}
		b_{11} & b_{12} \\
		b_{12} & b_{22}
		\end{pmatrix}
		\begin{pmatrix}
		0 & \frac{N}{2N-1} \\
		a_{N-1} & a_{N}
		\end{pmatrix}
		=
		\begin{pmatrix}
		b_{12}a_{N-1} & b_{11}\frac{N}{2N-1} + b_{12}a_N \\
		b_{22}a_{N-1} & b_{12}\frac{N}{2N-1} + b_{22}a_N
		\end{pmatrix}.		
	\end{equation}
	For $A_0 A$ to be symmetric, we need $BA_4$ is symmetric and $DA_2=(BA_3)^T$. Thus, we have the constraints
	\begin{equation}
	\begin{aligned}
		b_{11} \frac{N-1}{2N-1} + b_{12} a_{N-2} &= N-1, \\
		b_{12} \frac{N-1}{2N-1} + b_{22} a_{N-2} &= 0, \\
		b_{11}\frac{N}{2N-1} + b_{12} a_N &= b_{22} a_{N-1}.
	\end{aligned}
	\end{equation}
	We solve for $b_{11}$, $b_{12}$ and $b_{22}$ from the above linear system:
	\begin{equation}\label{eq:solution-b-linear-system}
	\begin{aligned}
		b_{11} &= \frac{(N-1)(2N-1)((N-1)a_{N-1} + (2N-1)a_{N-2}a_N)}{-N(2N-1)a_{N-2}^2 + (N-1)^2a_{N-1} + (N-1)(2N-1)a_{N-2}a_N}, \\
		b_{12} &= \frac{-N(N-1)(2N-1)a_{N-2}}{-N(2N-1)a_{N-2}^2 + (N-1)^2a_{N-1} + (N-1)(2N-1)a_{N-2}a_N}, \\
		b_{22} &= \frac{N(N-1)^2}{-N(2N-1)a_{N-2}^2 + (N-1)^2a_{N-1} + (N-1)(2N-1)a_{N-2}a_N}.
	\end{aligned}
	\end{equation}
	
	An equivalent condition for $B$ to be a SPD matrix is
	\begin{equation}
        b_{11} > 0, \quad b_{22}>0, \quad b_{11}b_{22}>b_{12}^2.
	\end{equation}
	From the expressions of $b_{11}$ and $b_{22}$ in \eqref{eq:solution-b-linear-system} and $b_{11}, b_{22}>0$ and $N\ge2$, we have
	\begin{equation}\label{eq:constraint-b11-positive}
		-N(2N-1)a_{N-2}^2 + (N-1)^2a_{N-1} + (N-1)(2N-1)a_{N-2}a_N > 0,		
	\end{equation}
	and
	\begin{equation}\label{eq:constraint-b22-positive}
		(N-1)a_{N-1} + (2N-1)a_{N-2}a_N >0.
	\end{equation}
	Plugging $b_{11}$, $b_{12}$ and $b_{12}$ in \eqref{eq:solution-b-linear-system} into $b_{11}b_{22}>b_{12}^2$, we further derive
	\begin{equation}\label{eq:constraint-b12-positive}
		(N-1)(2N-1)((N-1)a_{N-1} + (2N-1)a_{N-2}a_N) N(N-1)^2 > (N(N-1)(2N-1)a_{N-2})^2.
	\end{equation}
	It is easy to check that \eqref{eq:constraint-b11-positive} and \eqref{eq:constraint-b12-positive} are equivalent and \eqref{eq:constraint-b11-positive} implies \eqref{eq:constraint-b22-positive}. Thus, the constraint is
	\begin{equation}
		a_{N-1} > \frac{2N-1}{(N-1)^2} a_{N-2} \brac{N a_{N-2} - (N-1) a_N}.
	\end{equation}
	This is just the condition \eqref{eq:constraint-dof4} with $a_{N-3}=0$ (or equivalently \eqref{eq:constraint-dof3}). Finally, \eqref{eq:constraint-dof4-equivalent} can be easily derived, and the proof is omitted for brevity.
\end{proof}

\section{Structural stability condition}\label{sec:appendix-structure-stability}

In this part, we review the structural stability condition proposed in \cite{yong1999singular}, which is fundamental for the quasilinear first-order hyperbolic system with source terms:
\begin{equation}
	U_t + A(U) U_x = Q(U).
\end{equation}
Here $U=U(x,t)$ is an unknown $n$-vector valued function, $Q(U)$ and $A(U)$ are given $n$-vector and $n\times n$-matrix valued smooth functions of $U\in\mathcal{G}$ (an open subset of $\mathbb{R}^n$ called state space), respectively.
This structural stability condition was established for the hyperbolic system with source term in the multidimensional case. Here, we only show the condition in one-dimensional case for simplicity.

Define the equilibrium manifold
\begin{equation}
	\mathcal{E} = \{ U\in \mathcal{G} ~|~ Q(U) = 0 \}.
\end{equation}
The stability condition in \cite{yong1999singular} reads as:
\begin{enumerate}[label=(\roman*).]
	\item There exists an invertible $n\times n$ matrix $P(U)$ and an invertible $r\times r$ matrix $S(U)$, defined on the equilibrium manifold $\mathcal{E}$, such that
	\begin{equation}
		P(U)Q_U(U) = 
		\begin{pmatrix}
		0 & 0 \\
		0 & S(U)
		\end{pmatrix}
		P(U), \quad \textrm{for} ~ U\in\mathcal{E}.
	\end{equation}	

	\item There is a symmetric positive definite matrix $A_0(U)$ such that 
	\begin{equation}
		A_0(U) A(U) = A(U)^T A_0(U), \quad \textrm{for} ~ U\in\mathcal{G}.
	\end{equation}

	\item The hyperbolic part and the source term are coupled in the sense
	\begin{equation}
		A_0(U) Q_U(U) + Q_U^T(U) A_0(U) \le - P^T(U)
		\begin{pmatrix}
		0 & 0 \\
		0 & I_r
		\end{pmatrix}
		P(U), \quad \textrm{for} ~ U\in\mathcal{E}.
	\end{equation}
	Here $Q_U=\frac{\partial Q}{\partial U}$ and $I_r$ is the identity matrix of order $r$.
\end{enumerate}

This set of conditions has been tacitly respected by many well-developed physical models \cite{yong2008interesting}. The first condition is classical for initial value problems of the system of ordinary differential equations (ODEs), while the second one means the symmetrizable hyperbolicity of the PDE system. The third condition characterizes a kind of coupling between the convection term and the source term. This set of conditions implies the existence and stability of the zero relaxation limit of the corresponding initial value problems \cite{yong1999singular}.

\end{appendices}


\end{document}